\documentclass[final,onefignum,onetabnum,reqno]{siamart190516}

\usepackage{braket,amsfonts}

\usepackage{array}

\usepackage[caption=false]{subfig}
\usepackage{pgfplots}

\newsiamthm{claim}{Claim}
\newsiamremark{remark}{Remark}
\newsiamremark{hypothesis}{Hypothesis}
\crefname{hypothesis}{Hypothesis}{Hypotheses}

\usepackage{algorithmic}

\usepackage{graphicx,epstopdf}

\Crefname{ALC@unique}{Line}{Lines}

\usepackage{amsopn}

\usepackage{xspace}
\usepackage{bold-extra}
\usepackage[most]{tcolorbox}
\usepackage{multirow}

\colorlet{texcscolor}{blue!50!black}
\colorlet{texemcolor}{red!70!black}
\colorlet{texpreamble}{red!70!black}
\colorlet{codebackground}{black!25!white!25}

\lstdefinestyle{siamlatex}{%
  style=tcblatex,
  texcsstyle=*\color{texcscolor},
  texcsstyle=[2]\color{texemcolor},
  keywordstyle=[2]\color{texemcolor},
  moretexcs={cref,Cref,maketitle,mathcal,text,headers,email,url},
}

\tcbset{%
  colframe=black!75!white!75,
  coltitle=white,
  colback=codebackground, 
  colbacklower=white, 
  fonttitle=\bfseries,
  arc=0pt,outer arc=0pt,
  top=1pt,bottom=1pt,left=1mm,right=1mm,middle=1mm,boxsep=1mm,
  leftrule=0.3mm,rightrule=0.3mm,toprule=0.3mm,bottomrule=0.3mm,
  listing options={style=siamlatex}
}

\newtcblisting[use counter=example]{example}[2][]{%
  title={Example~\thetcbcounter: #2},#1}

\newtcbinputlisting[use counter=example]{\examplefile}[3][]{%
  title={Example~\thetcbcounter: #2},listing file={#3},#1}

\DeclareTotalTCBox{\code}{ v O{} }
{ 
  fontupper=\ttfamily\color{black},
  nobeforeafter,
  tcbox raise base,
  colback=codebackground,colframe=white,
  top=0pt,bottom=0pt,left=0mm,right=0mm,
  leftrule=0pt,rightrule=0pt,toprule=0mm,bottomrule=0mm,
  boxsep=0.5mm,
  #2}{#1}

\patchcmd\newpage{\vfil}{}{}{}
\begin{tcbverbatimwrite}{tmp_\jobname_header.tex}
\title{From Additive Average Schwarz Methods to \\ Non-overlapping Spectral Additive Schwarz Methods}

\author{YI YU\thanks{Mathematical Sciences, Worcester Polytechnic Institute, MA, USA(\email{yyu5@wpi.edu}).}
\and Maksymilian Dryja\thanks{Warsaw University, Warsaw, Poland (\email{dryja@mimuw.edu.pl}).}
\and Marcus Sarkis\footnotemark[3]\thanks{Mathematical Sciences, Worcester Polytechnic Institute, MA, USA(\email{msarkis@wpi.edu}).}}

\headers{From AAS to NOSAS}{YI YU, Maksymilian Dryja, and Marcus Sarkis}
\end{tcbverbatimwrite}
\input{tmp_\jobname_header.tex}

\ifpdf
\hypersetup{ pdftitle={From Additive Average Schwarz Methods to \\ Non-overlapping Spectral Additive Schwarz Methods} }
\fi

\begin{document}

\maketitle

\begin{tcbverbatimwrite}{tmp_\jobname_abstract.tex}
\begin{abstract}
 In this paper, we design and analyze two new methods based on additive average Schwarz -- AAS method introduced in \cite{MR1943457}. The new methods design for elliptic problems with highly heterogeneous coefficients. The methods are of the non-overlapping type, and the subdomain interactions obtain via the coarse space. The first method is the minimum energy Schwarz -- MES method. MES has the minimum energy for the coarse space with constant extension inside each subdomain. The condition number of the MES method is always smaller than in the AAS method. The second class of methods is the non-overlapping spectral additive Schwarz -- NOSAS methods based on low-rank discrete energy harmonic extension in each subdomain. To achieve the low-rank, we solve a generalized eigenvalue problem in each subdomain. NOSAS have the minimum energy for a given rank of the coarse space. The condition number of the NOSAS methods does not depend on the coefficients. Additionally, the NOSAS methods have good parallelization properties. The size of the global problem is equal to the total number of eigenvalues chosen in each subdomain. It is only related to the number of high-permeable islands that touch the subdomains' interface.  
\end{abstract}


\end{tcbverbatimwrite}
\input{tmp_\jobname_abstract.tex}

\section{INTRODUCTION}\hspace*{\fill} 

Let $\Omega \subset \mathbb{R}^{d=2 (3)}$ be a bounded polygonal (polyhedra) domain and let us impose homogeneous Dirichlet data on $\partial \Omega$.  Let us
introduce the Sobolev space:
\[
H_0^1(\Omega):=\{v \in H^1(\Omega);v=0 \text{ on } \partial \Omega\}.\\ 
\]
  The continuous variational formulation is given by: find $u \in H_0^1(\Omega)$ such that 
\begin{equation}\label{elliptic}
a(u,v)=f(v) \hspace{20pt} \mbox{for all}\hspace{20pt} v\in H_0^1(\Omega),
\end{equation}
where 
\begin{equation*}
     a(u,v):=\int_{\Omega} \rho(x)\nabla u\cdot\nabla vdx, \hspace{30pt} f(v):=\int_\Omega fv dx, 
\end{equation*}
and we assume $\rho(x) \ge \rho_{\min} > 0$ almost everywhere in $\Omega.$

The main purpose of this paper is to invent and analyze a class of domain decomposition preconditioners, denoted by non-overlapping spectral additive Schwarz -- NOSAS methods, for a finite element discretization for the problem \cref{elliptic}. To put these new algorithms into perspective, we first summarize some of the state-of-the-art domain decomposition preconditioners and show some of the differences with NOSAS methods. 

Any domain decomposition method base on the assumption that the given computation domain, say $\Omega$, is partitioned into subdomains $\Omega_i, 1\leq i\leq N$ which may or may not overlap, see \cite{MR2104179}. For the overlapping cases, in each iteration
of an iterative scheme, such as in the preconditioned conjugate gradient method, the idea is to solve local problems in $\Omega_i$ in parallel and combine them to construct a global approximation. For instance, if we add these local solutions, this method is called the additive Schwarz method. For the non-overlapping cases, the interface between these subdomains, say $\Gamma$, plays
a fundamental role; the idea is that if the solution of the discrete problem on the interface is known, the solution in the interior of the subdomains can be obtained easily and in parallel by solving a boundary value local problems. Hence, for the non-overlapping cases,
the algorithms target finding the solution on the interface. A way to do
that is via static condensation (via Schur complements where the interior unknowns to the subdomains are eliminated from the system). Overlapping Schwarz methods require some redundancy on the computations (larger overlap) to obtain fast convergence, while for the non-overlapping Schwarz methods, the Schur complement's application cannot be approximated. Some methods combine both strategies, such as the AAS method introduced in \cite{MR1943457,bjorstad2000domain} and analyzed in \cite{MR2770288} for a class of discontinuous coefficients.

 A coarse space is necessary for scalability with respect to the number of subdomains and the nature of the coefficients $\rho(x)$. The study of coarse spaces that make a preconditioner scalable with respect to the number of subdomains has a long history, see \cite{MR2104179}, where the main idea is to have a local Poincar\'e inequality on the space orthogonal to the coarse space. Only recently, coarse spaces were introduced to guarantee the robustness of the preconditioners for any coefficients $\rho(x)$, where the idea now is to have local weighted Poincar\'e inequality where the weights are related to the coefficients $\rho$ (which might be highly
heterogeneous). These coarse spaces nowadays are sometimes referred to as adaptive coarse spaces. The first non-empirical adaptive preconditioners were introduced to the non-overlapping case in \cite{MR2334131,MR2277024} to the BDDC and FETI-DP methods where mathematically motivated by the generalized eigenvalue problems based on energy operators. This work was later revisited in \cite{DP2013,MR3678572}  using
parallel and series sums techniques found in the classical article \cite{anderson1969series} and has become an important reference for the subject. For the overlapping cases, the first work was introduced in \cite{MR2718268,MR2728702} which also based on the generalized eigenvalue problems on the overlapping subdomains; this work also introduces the concept that the number of eigenvalues needed
to attain robustness is associated with the number of channels of
high conductivity crossing the subdomains boundaries. An incomplete list of references for more work for non-overlapping and overlapping cases can be seen in \cite{MR3612901,MR3582898,MR3350765,MR3350292,MR3303686,MR3739213,MR3089678} and \cite{MR3033238,MR3812062,MR4014789,MR3175183}, respectively, and references therein. 

The NOSAS are non-overlapping methods (with no redundancy on the computation) and without the need for Schur complement operators. A major difference between NOSAS and FETI-DP/BDDC is: NOSAS is based only on subdomains. For each subdomain $\Omega_i$ just one generalized eigenvalue problem is solved and using only the Neumann matrix associated to $\Omega_i$, while for BDDC/FETI-DP in 3D, the eigenvalues problems base on problems associated to edges and faces of the boundary of $\Omega_i$, that is, for each of this edge (face) of $\Omega_i$, all the Neumann matrices of the subdomains that share that edge (face) are needed. The methods that are closest
to our methods are the BDD-GenEO \cite{MR3093793} since they also are in solving one generalized eigenvalue problem per subdomain. However, the local matrices involved in this generalized eigenvalue problem require information from neighbors subdomains; the detailed differences of NOSAS and BDD-GenEO describe in \Cref{BDD-GenE0}. Another major difference is that NOSAS does not require a Schur complement operator while BDD-GeoEO/BDDC/FETI-DP does. Finally, we would like to mention the work \cite{MR3795491} which also uses an extension of the AAS method. Their method has a very different construction of the eigenvalue problems. While they use only information near the boundary of  $\Omega_i$, we use all information of $\Omega_i$. Consequently, we think our method is more efficient, more algebraic, and more natural to extend to other discretizations, including elasticity and other positive definite symmetric systems.

The remainder of this paper is organized as follows. \Cref{Section2}
describes the discretization and the domain decomposition framework (notation, subspaces, and operators). In \Cref{Section3} we introduce and analyze MES and NOSAS methods with the exact solver. To overcome the high complexity of the coarse problem in the exact solver,  we formulate and analyze NOSAS methods with the inexact solver in \Cref{Section4}.   Finally, \Cref{numerical} gives a global overview of the NOSAS method proposed with quantitative and qualitative studies.

\section{FINITE ELEMENT SPACES AND DOMAIN DECOMPOSITION} \label{Section2} 
\subsection{Discretization}\hspace*{\fill} 

We begin by discretizing \cref{elliptic} in an algebraic framework. Let us consider a triangulation $\mathcal{T}_h$ of $\Omega$ with
$\overline{\Omega}=\bigcup_{K\in\mathcal{T}_h }{K}$, where $K$ denotes a generic (closed) element of the triangulation. We assume that the partition $\mathcal{T}_h$ is shape regular and quasi-uniform of size $O(h)$. We can either require our partition fine enough such that the coefficients $\rho(x)$ is constant in each element $K$, denoted by $\rho_K$. Or we can choose a constant approximation of coefficients $\rho(x)$ in each element K, for example, $\rho_K=\int_K\rho(x)dx$. In this paper, the finite element space $V_h(\Omega)$ consists of continuous piecewise linear functions: 
\begin{equation*}
  V_h(\Omega):=\{v\in H^1_0(\Omega); v_{|{K}}\in P_1(K),\forall K \in \mathcal{T}_h\}=\text{Span}\{\phi_k; 1\leq k\leq n \}.
\end{equation*}
where  $n$ is the number of interior nodes of $\Omega$ and $\{\phi_k\}_{1\leq k\leq n}$ are basis functions, we also note that the techniques introducing in this paper also work for any polynomials of fixed order $p$.

Specifically, for any element $K\in \mathcal{T}_h$, let $V_h(K)=\{v_{|K};v\in V_h(\Omega)\}$, then for each $K$ there exists a symmetric positive semi-definite bilinear form $a_K: V_h(K)\times V_h(K)\to \mathbb{R}$, such that:
\begin{equation*}
    a(u,v)=\sum_{K \in \mathcal{T}_h}a_K(u_{|K},v_{|K}) \hspace{40pt}	\forall u,v\in V_h(\Omega),
\end{equation*}
and there exists an element $f_K\in V_h(K)'$ (the dual space of $V_h(K)$) such that:
\begin{equation*}
    f(v)=\sum_{K \in \mathcal{T}_h}f_K(v_{|K}) \hspace{55pt}	\forall v\in V_h(\Omega).
\end{equation*}
Then the FEM matrix form associated with \cref{elliptic} can be written as: find  $u_h\in {V}_h$, such that
\begin{equation} \label{matrixfm}
    Au_h=b,
\end{equation}
where the elements of $A$ and $b$ are defined as:
\begin{equation*}
    (A)_{kl}:=a(\phi_k,\phi_l)=\sum_{K \in \mathcal{T}_h}a_K(\phi_{k|K},\phi_{l|K})\hspace{20pt}	\forall 1 \leq k,l \leq n,
\end{equation*}
and
\begin{equation*}
    (b)_{k}:=f(\phi_k)=\sum_{K \in \mathcal{T}_h}f_K(\phi_{k|K})\hspace{50pt}	\forall 1 \leq k \leq n.
\end{equation*}
\subsection{Local setting}\hspace*{\fill} 

We decompose $\Omega$ into N non-overlapping open polygonal subdomains $\Omega_i$ of diameter $O(H)$ satisfy:
\begin{equation*}
    \overline{\Omega}=\bigcup_{i=1}^N\overline{\Omega}_i \hspace{10pt} \text{and} \hspace{10pt} \Omega_i\cap\Omega_j=\emptyset, \hspace{10pt}i\not=j.
\end{equation*}

We require that each subdomain is a union of shape regular triangular elements with nodes on the boundaries of neighboring subdomains matching across the interface.  We define the interface of each subdomain $\Omega_i$ by $\Gamma_i$
and also define global interface $\Gamma$ as:  
\begin{equation*}
    \Gamma_i := \partial \Omega_i 
\backslash \partial  \Omega, \quad \mbox{and} \quad    \Gamma:=\bigcup_{i=1}^N \Gamma_i.
\end{equation*}

\subsection{Decomposition of $V_h(\Omega)$}\hspace*{\fill} 

Let us define the local finite element space  
  $  V_h(\Omega_i)=\{v_{|{\Omega}_i}; \forall v\in V_h(\Omega) \}$
and consider a family of local spaces $\{V_i, 1\leq i \leq N\}$ on $V_h(\Omega_i)$ which vanish on $\partial \Omega_i$: 
 \begin{equation*}
     V_i := \{v \in V_h(\Omega_i) \text{ and } v=0 \text{ on } \partial \Omega_i    \},
 \end{equation*}
 and we define the extrapolation operators $R_{i}^{^T}:V_i\to V_h(\Omega)$ for $1\leq i\leq N $, where $R_{i}^{^T}$ are the extension by zero outside of $\Omega_i$. Note that $R_i: V_h(\Omega)\to V_i$ are the transpose of $R_i^T$, and $R_i$ are restriction operators which map a nodal vector on $\Omega$ to a nodal vector inside $\Omega_i$. 
      
The coarse space $V_0$ is defined as the restriction of $v\in V_h(\Omega)$ on $\Gamma$:
      \begin{equation*}
    V_0=V_h(\Gamma):=\{v_{|{\Gamma}}; \forall v\in V_h(\Omega)    \}.
      \end{equation*}

The core of this paper is to introduce and analyze different choices of the extrapolation operator $R_0^T:V_0\to V_h(\Omega)$; see \Cref{Section3}. 
  
Then $V_h(\Omega)$ admits the following direct sum decomposition: 
 \begin{equation*}
    V_h(\Omega)=R_0^TV_0\oplus R_{1}^{T}V_1 \oplus\cdots\oplus R_{N}^{T}V_N.  
\end{equation*}

\subsection{Local and coarse solvers}\label{solver}\hspace*{\fill} 

The Schwarz operators are constructed by the local and coarse solvers.\\
Local solvers: for $1 \leq i \leq N$, first we introduce the exact local bilinear form 
\begin{equation*}
    a_i(u_i,v_i):=v_i^TA_i\,u_i=a(R_{i}^{T}u_i,R_{i}^{T}v_i)  \hspace{20pt} \forall u_i,v_i \in V_i, 
\end{equation*}
the matrix form associated with $a_i(\cdot,\cdot)$ can be written as $A_i=R_iAR_i^T$.

Next we define the projection-like operator $T_i: V_h(\Omega)\to V_h(\Omega)$ be given by $T_i:=R_{i}^{T}\tilde{T}_i$,  where $\tilde{T}_i:V_h(\Omega)\to V_i$ defined as the local solver for the following local problem:
\begin{equation}\label{localpb1}
  a_i(\tilde{T}_iu_h,v_i)=a(u_h,R_{i}^{T}v_i) \hspace{25pt}\forall v_i\in V_i, \quad 1\leq i\leq N.
\end{equation}
Coarse solver: for $i=0$, first we consider the exact coarse bilinear form
\begin{equation*}
    a_0(u_{0},v_{0}):=v_0^TA_0\,u_0=a(R_0^Tu_{0},R_0^T v_{0})  \hspace{20pt} \forall u_{0},v_{0} \in V_0,
\end{equation*}
and let
$T_0: V_h(\Omega)\to V_h(\Omega)$ be given by $T_0:= R_{0}^{T}\tilde{T}_0$,  where $\tilde{T}_0:V_h(\Omega)\to V_0$ defined as the coarse solver for the following coarse problem:
\begin{equation}\label{localpb2}
    a_0(\tilde{T}_0u_h,v_{0})=a(u_h,R_0^Tv_{0}) \hspace{25pt} \forall v_{0}\in V_0.
\end{equation}
Note that we will also consider inexact bilinear form $\hat{a}_0(\cdot,\cdot)$ with respect to inexact solver later in this paper. The matrix form of $T_i$ above can be written as:
\begin{equation*}
    T_i=R_i^TA_i^{-1}R_iA, \quad\quad 0\leq i\leq N.
\end{equation*}

The traditional additive Schwarz methods-ASM are obtained by replacing linear system \cref{matrixfm} with the preconditioned system:
\begin{equation}
\label{Def_T_A} 
    T_Au_h=g_h \quad \mbox{where} \quad  T_A:=T_0+T_1+\cdots+T_N, \hspace{10pt}g_h=\displaystyle{\sum_{i=0}^Ng_i},
\end{equation}
where $g_i$ are obtained from right-hand side of \cref{localpb1} and \cref{localpb2}; see \cite{MR2104179}.

\subsection{Notations and operators}\hspace*{\fill} 

Let us put \cref{matrixfm} in the view of algebraic construction:
\begin{equation}\label{decomp_matrix}
   \begin{bmatrix}
   A_{\Gamma\Gamma}       & A_{\Gamma I}   \\
    A_{I \Gamma}     & A_{I I} 
\end{bmatrix}
\begin{bmatrix}
  u_{h_\Gamma}         \\
    u_{h_I}    
\end{bmatrix}
=   \sum_{i=1}^{N}R^{(i)^T}\begin{bmatrix}
   A_{\Gamma\Gamma}^{(i)}       & A_{\Gamma I}^{(i)}    \\
    A_{I \Gamma}^{(i)}      & A_{I I}^{(i)} 
\end{bmatrix}
R^{(i)}
\begin{bmatrix}
  u_{h_\Gamma}          \\
    u_{h_I}   
\end{bmatrix}
=\sum_{i=1}^{N}
R^{(i)^T}
\begin{bmatrix}
  b_{\Gamma}^{(i)}          \\
    b_{I}^{(i)}   
\end{bmatrix}.
\end{equation}
Here $u_{h_\Gamma}$ and $u_{h_I}$ are the restriction of $u_h$ on $\Gamma$  and $I:=\Omega\backslash\Gamma$ respectively. $b_\Gamma^{(i)}$, $b_I^{(i)}$ are defined below. The extrapolation operators $R^{(i)^T}:V_h({\Omega}_i )\to V_h(\Omega)$ for $1\leq i\leq N $ are  the extension by zero outside of $\overline{\Omega}_i$. Note that $R^{(i)}: V_h(\Omega)\to V_h({\Omega}_i) $ are the transpose of $R^{(i)^T}$, and $R^{(i)}$ are restriction operators which map a nodal vector  on $\Omega$ to a nodal vector on $\overline{\Omega}_i$. We note that $A_i=A_{II}^{(i)}$. Throughout the paper, we use $A_i$ when talking about the local solver, and use $A_{II}^{(i)}$ when dealing the coarse solver due to Schur complement notation.

Thus we have,
\begin{equation*}
    A=\sum_{i=1}^{N}R^{(i)^T}A^{(i)}R^{(i)} \hspace{20pt}\text{and}\hspace{20pt} b=\begin{bmatrix}
  b_\Gamma          \\
    b_I   
\end{bmatrix}=\sum_{i=1}^{N}
R^{(i)^T}
\begin{bmatrix}
  b_{\Gamma}^{(i)}          \\
    b_{I}^{(i)}   
\end{bmatrix},
\end{equation*}
where $A^{(i)}$ is the Neumann matrix corresponding to the bilinear form of 
\begin{equation*}
  a^{(i)}(u_i,v_i)  =\sum_{K \in \mathcal{T}_{{h}|\bar{\Omega}_i}}a_K(u_{i|K},v_{i|K}) 
\hspace{20pt} \forall u_i,v_i \in V_h(\Omega_i),
\end{equation*}
and $b_\Gamma^{(i)}$, $b_I^{(i)}$ are the restriction of $b^{(i)}$ on $\Gamma_i$ and inside $\Omega_i$ respectively,
\begin{equation*}
    \begin{bmatrix}
  b_\Gamma^{(i)}          \\
    b_I^{(i)}  
\end{bmatrix}=b^{(i)}=\sum_{K \in \mathcal{T}_{{h}|\bar{\Omega}_i}}f_K(v_{i|K}) 
\hspace{20pt} \forall v_i \in V_h(\Omega_i).
\end{equation*}
Let us define $V_h(\Gamma_i)=\{v_{|{\Gamma}_i}; \forall v\in V_h(\Omega) \}$, and $ R_{\Gamma_i}^{T}:V_h(\Gamma_i)\to V_h(\Omega)$ is the extension by zero outside of  $\Gamma_i$.  Correspondingly, $R_{\Gamma_i}:V_h(\Omega)\to V_h(\Gamma_i)$  is the restriction operator which map a nodal vector on $\Omega$ to a nodal vector on ${\Gamma}_i$. 

Moreover, if we always first label the interface nodes and then label the interior nodes, we can always decompose the Boolean matrix $R^{(i)^T}$ as:
\begin{equation*}
    R^{(i)^T}=[R_{\Gamma_i}^{T},R_{i}^{T}]=\begin{bmatrix}
    R_{\Gamma_i\Gamma}^{T}     & 0    \\
    0     &   R_{I_iI}^{T}
    \end{bmatrix}
    \quad \mbox{and} \quad
\begin{bmatrix}
  u_{h_\Gamma}^{(i)}          \\
    u_{h_I}^{(i)}   
\end{bmatrix} = R^{(i)} 
\begin{bmatrix}
  u_{h_\Gamma}          \\
    u_{h_I}   
\end{bmatrix},
    \end{equation*}
where  $ R_{\Gamma_i\Gamma}^{T}:V_h(\Gamma_i)\to V_h(\Gamma)$ is the extension by zero outside of $\Gamma_i$, and $R_{I_iI}^{T}:V_i\to V_h(I)$
is the extension by zero outside of $\Omega_i$ and 
 $V_h(I):=\{v_{|{I}}; \forall v\in V_h(\Omega) \}$. Note that $R_{\Gamma_i\Gamma}: V_h(\Gamma)\to V_h(\Gamma_i)$ is the restriction operator which map a nodal vector on $\Gamma$ to a nodal vector on ${\Gamma}_i$, and $R_{I_iI}:V_h(I)\to V_i$ is  the restriction operator which map a nodal vector on $I$ to a nodal vector on the interior of ${\Omega}_i$. The $u_{h_\Gamma}^{(i)}$ is the restriction of $u_{h_\Gamma}$ on $\Gamma_i$, and $u_{h_I}^{(i)} $ is the restriction of  $u_{h_I}$ inside $\Omega_i$.

We now rewrite \cref{decomp_matrix} as Schur complement system:
\begin{equation*}
\small
   \begin{bmatrix}
   S      & 0   \\
    A_{I \Gamma}     & A_{I I} 
\end{bmatrix}
\begin{bmatrix}
  u_{h_\Gamma}         \\
    u_{h_I}    
\end{bmatrix}
=   \sum_{i=1}^{N}R^{(i)^T}\begin{bmatrix}
   S^{(i)}       & 0    \\
    A_{I \Gamma}^{(i)}      & A_{I I}^{(i)} 
\end{bmatrix}
\begin{bmatrix}
  u_{h_\Gamma}^{(i)}          \\
    u_{h_I}^{(i)}   
\end{bmatrix}
=\sum_{i=1}^{N}
R^{(i)^T}
\begin{bmatrix}
  b_{\Gamma}^{(i)}-     A_{\Gamma I}^{(i)}(A_{II}^{(i)})^{-1} b_{I}^{(i)}      \\
    b_{I}^{(i)}   
\end{bmatrix}
=\begin{bmatrix}
  	\tilde{b}_{\Gamma}      \\
    	{b}_{I}  
\end{bmatrix},
\end{equation*}
where 
\begin{equation*}
   S^{(i)}=A_{\Gamma\Gamma}^{(i)}-A_{\Gamma I}^{(i)}(A_{II}^{(i)})^{-1}A_{I\Gamma }^{(i)},
\end{equation*}
\begin{equation*}
    \tilde{b}_{\Gamma}=\sum_{i=1}^{N}R_{\Gamma_i\Gamma}^{T}   (b_{\Gamma}^{(i)}-A_{\Gamma I}^{(i)}(A_{II}^{(i)})^{-1} b_{I}^{(i)})  \hspace{20pt}\text{and}\hspace{20pt}S=\sum_{i=1}^{N}R_{\Gamma_i\Gamma}^{T} S^{(i)}R_{\Gamma_i\Gamma},
\end{equation*}
$R_{\Gamma_i\Gamma}^{T} $ is defined as before, and the reduced system is given by:
\begin{equation}
    Su_{h_\Gamma}=\tilde{b}_{\Gamma}.
\end{equation}
The goal for non-overlapping additive Schwarz methods is to find a good preconditioner $S_0^{-1}$ for S. See \cite{MR1943457},\cite{bjorstad2000domain},\cite{MR2770288},\cite{MR3795491}.

\section{PRECONDITIONERS: OLD AND NEW ONES} \label{Section3}\hspace*{\fill} 

In this section, we will introduce different preconditioners based on the definition of $R_0^T$. Above and throughout, for all $u\in V_h(\Omega)$, we denote $u_{\Gamma_i}\!\!=\!R_{\Gamma_i\Gamma}u_\Gamma\!=R_{\Gamma_i}u$ is the restriction of $u$ on $\Gamma_i$, where $u_\Gamma\in V_0$ is the restriction of $u$ on $\Gamma$.
And denote $\Gamma_{ih}$ the sets of nodal points on $\Gamma_i$, and $ \Omega_{ih}$ the sets of interior nodal points in $ \Omega_{i}$. 

\subsection{Method 1: Harmonic extension (the optimal choice)}\hspace*{\fill} 

 Let us define the local $a_i$-discrete harmonic extension operator $\mathcal{H}^{(i)}: V_h( \Gamma_i)  \to V_h(\Omega_i)$ as:
\begin{equation*}
    \mathcal{H}^{(i)}u_{\Gamma_i} :=\begin{cases}\hspace{30pt} u_{\Gamma_i} \hspace{65pt}  \mbox{on}~~  \Gamma_{ih}, \\
    -(A_{II}^{(i)})^{-1}A^{(i)}_{I\Gamma}u_{\Gamma_i}\hspace{35pt} \mbox{in} ~~ \Omega_{ih}.
    \end{cases}
\end{equation*}
We remind that $V_h(\Omega_i)$ assumes zero Dirichlet condition on $\partial\Omega_i\cap\partial \Omega$.
And define the global $a$-discrete harmonic extension operator $\mathcal{H}: V_0 \to  V_h(\Omega)$ as:
\begin{equation*}
    \mathcal{H} u_\Gamma :=\begin{bmatrix}u_\Gamma \\
    \displaystyle{-\!\sum_{i=1}^N}R_{I_iI}^{T}(A_{II}^{(i)})^{-1}A^{(i)}_{I\Gamma}R_{\Gamma_i\Gamma}u_\Gamma
    \end{bmatrix}.
\end{equation*}
Note that the bilinear form of Schur complement $s(\cdot,\cdot)$ can be defined as:
\begin{equation*}
  s(v_\Gamma,u_\Gamma)\!=\! v_\Gamma^T S u_\Gamma\!\!=\! v_\Gamma^T\sum_{i=1}^NR_{\Gamma_i\Gamma}^{T}  (A^{(i)}_{\Gamma \Gamma }-A^{(i)}_{\Gamma I}(A_{II}^{(i)})^{-1}A^{(i)}_{I\Gamma})R_{\Gamma_i\Gamma}  u_\Gamma\!\! =\! a(\mathcal{H}u_\Gamma,\mathcal{H}v_\Gamma ),  \hspace{5pt} \forall u_\Gamma,v_\Gamma\in V_0.
\end{equation*}
Thus
\begin{equation*}
  S = \sum_{i=1}^N R_{\Gamma_i\Gamma}^{T}  (A^{(i)}_{\Gamma \Gamma }-A^{(i)}_{\Gamma I}(A_{II}^{(i)})^{-1}A^{(i)}_{I\Gamma})R_{\Gamma_i\Gamma} =\sum_{i=1}^N R_{\Gamma_i\Gamma}^{T}
   S^{(i)} R_{\Gamma_i\Gamma} = {\mathcal{H}}^T A \mathcal{H}.
\end{equation*}

If we choose extrapolation operator $R_0^T = \mathcal{H}$, that is, the $a$-discrete harmonic extension, the corresponding ASM is a direct solver. In this case, $S_0^{-1} = S^{-1}$. Of course, this preconditioner is too expensive, therefore, we consider better ones. 

\subsection{Method 2: additive average Schwarz method (AAS)}\hspace*{\fill} 

For AAS, let us select one degree of freedom per each subdomain (we choose a constant function as representative). In order to define
a global extension from $V_0\to V_h(\Omega)$, let us first construct for each
subdomain $\Omega_i$ the distribution mappings $P^{(i)}: \mathbb{R} \to V_i$ and
 $Q^{(i)^T}: V_h(\Gamma_i) \to \mathbb{R}$ as: 
\begin{equation*}
    P^{(i)}= \begin{bmatrix}
  1 \\ 1  \\\vdots \\  1
    \end{bmatrix}
    \quad \mbox{and define} \quad  Q^{(i)}
    =     \frac{1}{m_i}
\begin{bmatrix}
  1  \\ \vdots \\  1
    \end{bmatrix},
\end{equation*}
where $m_i$ is the number of nodal points on $\partial\Omega_{i}$. For any $u_{\Gamma_i}\in V_h(\Gamma_i) $, see that
$Q^{(i)^T} u_{\Gamma_i}$ is the nodal value average of $u_{\Gamma_i}$ on $\partial\Omega_{i}$ rather than on $\Gamma_{i}$ (we note that better numerical results are obtained by averaging on
$\partial\Omega_{i}$ rather than on $ \Gamma_{i}$ for constant coefficients in each subdomain). Then we construct a global vector which consists of all the averages of each subdomain, denoted by 
\begin{equation*}
\bar{u}=\begin{bmatrix}
  \bar{u}_1  \\ \vdots \\  \bar{u}_N
    \end{bmatrix},
\end{equation*}
where $\bar{u}_i$ is the constant function in subdomain $\Omega_i$ for $1\leq i\leq N$.

Let us define the local extension operator ${\hat{\mathcal{H}}}^{(i)}: V_h(\Gamma_i)  \to V_h(\Omega_i)$ as:
\begin{equation*}
    {\hat{\mathcal{H}}}^{(i)}  u_{\Gamma_i} :=\begin{cases}\hspace{20pt} u_{\Gamma_i} \hspace{54pt} \mbox{on}~~  \Gamma_{ih}, \\
    P^{(i)}Q^{(i)^T} u_{\Gamma_i} \hspace{35pt} \mbox{in} ~~ \Omega_{ih}.
    \end{cases}
\end{equation*}
and let define the global extension operator $\hat{\mathcal{H}}: V_0 \to  V_h(\Omega)$ as:
\begin{equation*}
   \hat{\mathcal{H}}  u_\Gamma :=\begin{bmatrix}
  u_\Gamma  \\
    \displaystyle{\sum_{i=1}^N}R_{I_iI}^{T}P^{(i)}Q^{(i)^T} R_{\Gamma_i\Gamma} u_\Gamma 
 \end{bmatrix}.
\end{equation*}
We let $R_0^T = \hat{\mathcal{H}}$ and the 
coarse bilinear form for the AAS is constructed as:
\begin{equation*}
    {a}_0(u_{\Gamma},v_{\Gamma})=a(\hat{\mathcal{H}}u_{\Gamma},\hat{\mathcal{H}}v_{\Gamma})=\sum_{i=1}^Na^{(i)}(\hat{\mathcal{H}}^{(i)}u_{\Gamma_i},\hat{\mathcal{H}}^{(i)}v_{\Gamma_i})\hspace{20pt} \forall u_{\Gamma},v_{\Gamma}\in V_0,
\end{equation*}
or the inexact bilinear form of AAS
\begin{equation*}
    \tilde{a}_0(u_{\Gamma},v_{\Gamma})=\sum_{i=1}^N\sum_{x\in \partial\Omega_{ih}}{\rho}_i(u_{\Gamma_i}(x)-\bar{u}_i)(v_{\Gamma_i}(x)-\bar{v}_{i})\hspace{20pt} \forall u_{\Gamma},v_{\Gamma}\in V_0,
\end{equation*}
where $\partial\Omega_{ih}$ are the sets of nodal points on $\partial \Omega_i$ with $u_\Gamma(x)=0$ when $x\in \partial\Omega$, $\bar{u}_i =  Q^{(i)^T} u_{\Gamma_i}$ and $\bar{v}_i =  Q^{(i)^T} v_{\Gamma_i}$.
Let us define $\Omega_i^\delta\subset\Omega_i$  be the layer around $\partial\Omega_i$ with one element width. Notice that the traditional AAS is robust when the coefficients is constant on $\Omega_i^\delta$, see \cite{MR1943457}. 

\subsection{New method: minimum energy Schwarz method (MES)}\hspace*{\fill} 

Instead of defining $\bar{u}_i =  Q^{(i)^T} u_{\Gamma_i}$, we try to define
a better weighting of $u_{\Gamma_i}$ to take into account the coefficients
in $\Omega_i$.

Let us introduce $\mathcal{I}_i: V_h(\Gamma_i) \times  \mathbb{R} \to V_h(\Omega_i)$ as: 
\begin{equation*}
   \mathcal{I}_i (u_{\Gamma_i},\bar{v}_i)  :=\begin{cases}\hspace{5pt}u_{\Gamma_i} \hspace{16pt} \mbox{on}~~  \Gamma_{ih}, \\
    P^{(i)}\bar{v}_i \hspace{11pt} \mbox{in} ~~ \Omega_{ih}.
    \end{cases}
\end{equation*}
Here the mappings $P^{(i)}: \mathbb{R} \to V_i$ is the same in the previous section. Now instead of defining $\bar{u}_i =  Q^{(i)^T} u_{\Gamma_i}$, the idea now
is to choose $\bar{u}_i$ such that
\begin{equation*}
  a^{(i)}\big(\mathcal{I}_i(u_{\Gamma_i},\bar{u}_i),\mathcal{I}_i(u_{\Gamma_i},\bar{u}_i)\big)  
  = \min_{\bar{v}_i \in  \mathbb{R}} a^{(i)}\big(\mathcal{I}_i(u_{\Gamma_i},\bar{v}_i),\mathcal{I}_i(u_{\Gamma_i},\bar{v}_i)\big),
  \end{equation*}
and define  $ \tilde{\mathcal{H}}^{(i)}: V_h(\Gamma_i)  \to V_h(\Omega_i)$ as: 
\begin{equation*}
   \tilde{\mathcal{H}}^{(i)} u_{\Gamma_i} :=\begin{cases}\hspace{5pt}u_{\Gamma_i} \hspace{16pt} \mbox{on}~~  \Gamma_{ih}, \\
    P^{(i)}\bar{u}_i \hspace{11pt} \mbox{in} ~~ \Omega_{ih}.
    \end{cases}
\end{equation*}

The idea is that the constant $\bar{u}_i$ is chosen so that the
extension of ${u}_{\Gamma_i}$ by a constant value inside $\Omega_{i}$
has the minimum energy.  Hence, taking the derivative with respect to $\bar{u}_i$, we have $P^{(i)^T}\!\!(A^{(i)}_{I\Gamma}u_{\Gamma_i}+A^{(i)}_{II}P^{(i)}\bar{u}_i)\!\!=\!0$,
therefore, $\bar{u}_i\!=\!\!-(P^{(i)^T}A_{II}^{(i)}P^{(i)})^{-1}P^{(i)^T}\!\!A_{I\Gamma}^{(i)}u_{\Gamma_i}$, and
\begin{equation*}
    \tilde{\mathcal{H}}^{(i)}u_{\Gamma_i}  :=\begin{cases}
\hspace{60pt}u_{\Gamma_i}  \hspace{110pt} \mbox{on}~~ \Gamma_{ih},\\
 -P^{(i)}(P^{(i)^T}A_{II}^{(i)}P^{(i)})^{-1}P^{(i)^T}A_{I\Gamma}^{(i)}u_{\Gamma_i} \hspace{30pt} \mbox{in}~~\Omega_{ih}.
\end{cases} 
\end{equation*}
and let define global extension operator $\tilde{\mathcal{H}}: V_0 \to  V_h(\Omega)$ as:
\begin{equation*}
   \tilde{\mathcal{H}}  u_\Gamma :=\begin{bmatrix}u_\Gamma \\
    \displaystyle{-\sum_{i=1}^N}R_{I_iI}^{T}P^{(i)}(P^{(i)^T}A_{II}^{(i)}P^{(i)})^{-1}P^{(i)^T}A_{I\Gamma}^{(i)} R_{\Gamma_i\Gamma} u_\Gamma 
    \end{bmatrix}.
\end{equation*}
We let  $R_0^T = \tilde{\mathcal{H}}$ and the exact bilinear form given as:
\begin{small}
\begin{equation*}
\begin{split}
     a_0(u_\Gamma,v_\Gamma)= a(\tilde{\mathcal{H}}u_\Gamma,\tilde{\mathcal{H}}v_\Gamma)
   &=\sum_{i=1}^N a^{(i)}(\tilde{\mathcal{H}}^{(i)}u_{\Gamma_i},\tilde{\mathcal{H}}^{(i)}u_{\Gamma_i})\\
   &=v_\Gamma^T\sum_{i=1}^NR_{\Gamma_i\Gamma}^{T}\big(A^{(i)}_{\Gamma \Gamma }\!-\!A^{(i)}_{\Gamma I}P^{(i)}(P^{(i)^T}A^{(i)}_{II}P^{(i)})^{-1}P^{(i)^T}A^{(i)}_{I\Gamma}\big)R_{\Gamma_i\Gamma}u_\Gamma\\
    &=v_\Gamma^T\big(A_{\Gamma \Gamma }\!-\!A_{\Gamma I}P(P^TA_{II}P)^{-1}P^TA_{I\Gamma}\big)u_\Gamma \hspace{30pt} u_\Gamma,v_\Gamma\in V_0.\\
\end{split}
\end{equation*}
\end{small}
In the above equation we use the global assembling matrices:  
\begin{equation*}
    A_{\Gamma\Gamma}=\sum_{i=1}^NR_{\Gamma_i\Gamma}^{T}A_{\Gamma\Gamma}^{(i)}R_{\Gamma_i\Gamma},\hspace{5pt}A_{I\Gamma}=\sum_{i=1}^NR_{I_iI}^{T}A_{I\Gamma}^{(i)}R_{\Gamma_i\Gamma},\hspace{5pt}\text{and}\hspace{3pt} P=\sum_{i=1}^NR_{I_iI}^{T}P^{(i)}R_{\bar{u}}^{(i)},
\end{equation*}
where  $R_{\bar{u}}^{(i)}: \bar{u} \to \mathbb{R}$ is the restriction choosing the i-th entry from the $N\times 1$ vector $\bar{u}$.\\
The following two lemmas show the condition number of the MES methods. We denote $|\cdot|$ as seminorm, and $|| \cdot|| $ as full norm, and  we write $a\preceq b$ when there exists constant $C>0$, independent of $\rho$, $h$ and $H$, it depends only on the shape of the elements and the shape of the subdomain, such that $a\leq Cb$. We also write $a\asymp b$ if $a \preceq b$ and $b	\preceq a$.

\begin{lemma}
\label{3.3.1}
In i-th subdomain, if $\rho_K(x)\equiv \rho_i$ for all elements $K\subset \overline{\Omega}_i$, then \begin{equation*}
    a^{(i)}(\tilde{\mathcal{H}}^{(i)}u_{\Gamma_i},\tilde{\mathcal{H}}^{(i)}u_{\Gamma_i})\preceq\frac{H}{h}a^{(i)}(\mathcal{H}^{(i)}u_{\Gamma_i},\mathcal{H}^{(i)}u_{\Gamma_i}),
\end{equation*} where $\tilde{\mathcal{H}}^{(i)}$ is the local MES extension, and $\mathcal{H}^{(i)}$ is the local $a_i$-discrete harmonic extension.  
\end{lemma}
\begin{proof}
For AAS method we defined $ {\hat{\mathcal{H}}}^{(i)}  u_{\Gamma_{i}} = u_{\Gamma_{i}}$ on $\Gamma_{ih}$ and
    $ {\hat{\mathcal{H}}}^{(i)}  u_{\Gamma_{i}} = t$ in $\Omega_{ih}$ with $ t =  \frac{\int_{\Gamma_i} u_{\Gamma_i}\, dx}{\int_{\partial\Omega_i}1\, dx}$. Remember that we denoted $\Omega_i^\delta\subset\Omega_i$ to be the layer around $\partial \Omega_i$ with one element width. Then by the definition of $\tilde{\mathcal{H}}^{(i)}$, we have 
   \begin{equation*}
   a^{(i)}(\tilde{\mathcal{H}}^{(i)}u_{\Gamma_i},\tilde{\mathcal{H}}^{(i)}u_{\Gamma_i})\leq a^{(i)}({\hat{\mathcal{H}}}^{(i)}  u_{\Gamma_{i}}, {\hat{\mathcal{H}}}^{(i)}  u_{\Gamma_{i}} )=\rho_i| {\hat{\mathcal{H}}}^{(i)}  u_{\Gamma_{i}}     -t|^2_{H^1{(\Omega_i^\delta)}}.
 \end{equation*} 
 Using the inverse inequality and notice that $ {\hat{\mathcal{H}}}^{(i)} u_{\Gamma_{i}} - t=0$ in $\Omega_{ih}$, we have
  \begin{equation*}
    \rho_i| {\hat{\mathcal{H}}}^{(i)}  u_{\Gamma_{i}}  -t|^2_{H^1{(\Omega_i^\delta)}}\preceq \rho_i\frac{1}{h^2}\|    {\hat{\mathcal{H}}}^{(i)}  u_{\Gamma_{i}} -t\|^2_{L^2{(\Omega_i^\delta)}}\preceq \rho_i\frac{1}{h}\|u_{\Gamma_{i}}-t\|^2_{L^2{(\partial\Omega_i)}}.
 \end{equation*}
Then use Poincar\'e inequality on $\partial \Omega_i$ and properties of $H^{1/2}$ norm 
  \begin{equation*}
  \begin{split}
    \rho_i \frac{1}{h}\|u_{\Gamma_{i}}-t\|^2_{L^2{(\partial\Omega_i)}}&\preceq\rho_i   \frac{H}{h}  |u_{\Gamma_{i}}|_{H^{1/2}(\partial \Omega_i)} \preceq \frac{H}{h}a^{(i)}(\mathcal{H}^{(i)}u_{\Gamma_i},\mathcal{H}^{(i)}u_{\Gamma_i}).
\end{split}
 \end{equation*}
\end{proof}

Using \Cref{3.3.1}, we can show if the coefficients $\rho(x)$ is constant inside each subdomain, the condition number of MES is always smaller than AAS, and is $O(H/h)$. Moreover, we can prove that the condition number of MES is only associated with the coefficients on $\Gamma_i^\delta$, where
 $\Gamma_i^\delta$ defined as the union of all elements of $\Omega_{i}$  which touch at least one node of $\Gamma_{i}$. 
\begin{lemma}
  \label{3.3.2}
  In i-th subdomain, if $\displaystyle{\sup_{K\subset\Gamma_i^\delta}\rho_K=\overline{\rho}_i}$ and $\displaystyle{\inf_{K\subset\Gamma_i^\delta}\rho_K=\underline{\rho}_i}$, then
  \begin{equation*}
    a^{(i)}(\tilde{\mathcal{H}}^{(i)}u_{\Gamma_i},\tilde{\mathcal{H}}^{(i)}u_{\Gamma_i})\preceq \frac{\overline{\rho}_i}{\underline{\rho}_i}\frac{H^2}{h^2}a^{(i)}(\mathcal{H}^{(i)}u_{\Gamma_i},\mathcal{H}^{(i)}u_{\Gamma_i}),
\end{equation*} where $\tilde{\mathcal{H}}^{(i)}$ and $\mathcal{H}^{(i)}$ are defined above. 
\end{lemma}
\begin{proof}
Since the subdomain has two types, the floating subdomain (not touches $\partial \Omega$) and the subdomain that touches $\partial \Omega$, we discuss two types separately. First, for the floating subdomain, since $\Gamma_i=\partial\Omega_i$ and $\Gamma_i^\delta=\Omega_i^\delta$, we use the AAS method and similar arguments as the previous lemma, we have
   \begin{equation*}
     a^{(i)}(\tilde{\mathcal{H}}^{(i)}u_{\Gamma_i},\tilde{\mathcal{H}}^{(i)}u_{\Gamma_i})\leq a^{(i)}({\hat{\mathcal{H}}}^{(i)}  u_{\Gamma_{i}}, {\hat{\mathcal{H}}}^{(i)}  u_{\Gamma_{i}})\leq \overline{\rho}_i |{\hat{\mathcal{H}}}^{(i)}  u_{\Gamma_{i}} -t|^2_{H^1{(\Gamma_i^\delta)}}.
 \end{equation*} 
   Using the inverse inequality and the fact that ${\hat{\mathcal{H}}}^{(i)}  u_{\Gamma_{i}} -t=0$ in $\Omega_{ih}$, we have
 \begin{equation*}
   \overline{\rho}_i |{\hat{\mathcal{H}}}^{(i)}  u_{\Gamma_{i}} -t|^2_{H^1{(\Gamma_i^\delta)}}\preceq \overline{\rho}_i \frac{1}{h^2}\|{\hat{\mathcal{H}}}^{(i)}  u_{\Gamma_{i}}-t\|^2_{L^2{(\Gamma_i^\delta)}}\preceq \overline{\rho}_i \frac{1}{h}\|u_{\Gamma_i}-t\|^2_{L^2{(\Gamma_i)}}.
 \end{equation*}
 Then use Poincar\'e inequality on $\Gamma_i$,
  \begin{equation*}
   \overline{\rho}_i \frac{1}{h}\|u_{\Gamma_i}-t\|^2_{L^2{(\Gamma_i)}}\preceq  \overline{\rho}_i\frac{H^2}{h}|u_{\Gamma_i}-t|^2_{H^1{(\Gamma_i)}}.
 \end{equation*}
 By an element-wise argument (see \cite{dryja2011technical}), we have  
 \begin{equation*}
 \small
  \overline{\rho}_i\frac{H^2}{h}|u_{\Gamma_i}\!-t|^2_{H^1{\!(\Gamma_i)\!}}\!\preceq\overline{\rho}_i \frac{H^2}{h^2}|\mathcal{H}^{(i)}\!u_{\Gamma_i}\!-t|^2_{H^1{\!(\Gamma_i^\delta)\!}}
    \! = \overline{\rho}_i \frac{H^2}{h^2}|\mathcal{H}^{(i)}\!u_{\Gamma_i}|^2_{H^1{\!(\Gamma_i^\delta)\!}}
     \leq \frac{\overline{\rho}_i}{\underline{\rho}_i}\!\frac{H^2}{h^2}a^{(i)}\!(\mathcal{H}^{(i)}\!u_{\Gamma_i}\!,\!\mathcal{H}^{(i)}\!u_{\Gamma_i}).
\end{equation*}
For the subdomain that touch $\partial \Omega$, we define $ {{\mathcal{E}}}^{(i)}  u_{\Gamma_{i}} = u_{\Gamma_{i}}$ on $\Gamma_{ih}$ and
    $ {{\mathcal{E}}}^{(i)}  u_{\Gamma_{i}} = 0$ on $\partial\Omega$ and in $\Omega_{ih}$.The above arguments hold true if we replace ${\hat{\mathcal{H}}}^{(i)}  u_{\Gamma_{i}} $ by ${\mathcal{E}}^{(i)} u_{\Gamma_{i}}$.
    
    \begin{equation*}
     a^{(i)}(\tilde{\mathcal{H}}^{(i)}u_{\Gamma_i},\tilde{\mathcal{H}}^{(i)}u_{\Gamma_i})\leq a^{(i)}({\mathcal{E}}^{(i)}  u_{\Gamma_{i}}, {\mathcal{E}}^{(i)}  u_{\Gamma_{i}})\leq \overline{\rho}_i |{\mathcal{E}}^{(i)}  u_{\Gamma_{i}} |^2_{H^1{(\Gamma_i^\delta)}},
 \end{equation*} 
 since ${\mathcal{E}}^{(i)}  u_{\Gamma_{i}}$ vanishes on $\Omega_i\backslash\Gamma_i^\delta$. Similarly we have,
 \begin{equation*}
   \overline{\rho}_i |{{\mathcal{E}}}^{(i)}  u_{\Gamma_{i}} |^2_{H^1{(\Gamma_i^\delta)}}\preceq \overline{\rho}_i \frac{1}{h}\|u_{\Gamma_i}\|^2_{L^2{(\Gamma_i)}}\preceq \overline{\rho}_i \frac{H^2}{h}|u_{\Gamma_i}|^2_{H^1{(\Gamma_i)}}\preceq  \frac{\overline{\rho}_i}{\underline{\rho}_i}\frac{H^2}{h^2}a^{(i)}(\mathcal{H}^{(i)}u_{\Gamma_i},\mathcal{H}^{(i)}u_{\Gamma_i}),
 \end{equation*}
 where we use Poincar\'e inequality on $\partial \Omega_i$ since $ {{\mathcal{H}}}^{(i)}  u_{\Gamma_{i}} = 0$ on $\partial \Omega\cap \partial \Omega_i$.

\end{proof}

 \Cref{3.3.2} shows that if $\overline{\rho}_i=\underline{\rho}_i$, which means that the coefficients $\rho(x)$ is constant in $\Gamma_i^\delta$, then we can always expect MES have condition number $O(H^2/h^2)$. Moreover, \Cref{3.3.2} estimates the worst coefficients scenario. In some special cases, which are discussed in \cref{numerical}, even if we have high-contrast coefficients in $\Gamma_i^\delta$, MES still work well. But in order to handle the general situation where the coefficients have high-contrast in $\Gamma_i^\delta$, we introduce the following new method. 

\subsection{New method: non-overlapping spectral additive Schwarz method with exact solver(NOSAS)}\hspace*{\fill} \label{NewMethod} 

First, we study the following local generalized eigenvalue problem in each subdomain $(i=1,\cdots, N)$ separately:
\begin{equation}\label{exact_eigen}
 S^{(i)} \xi_j^{(i)} :=  (A^{(i)}_{\Gamma\Gamma}-A_{\Gamma I}^{(i)}(A_{II}^{(i)})^{-1}A^{(i)}_{I \Gamma })\xi_j^{(i)}=\lambda^{(i)}_j A_{\Gamma\Gamma}^{(i)}\xi^{(i)}_j,
\hspace{30pt} (j=1,\cdots, n_i)\end{equation}
where $n_i$ is the degrees of freedom on $\Gamma_i$, and $0\leq\lambda_1^{(i)} \leq \cdots \leq \lambda_{n_i}^{(i)} \leq 1$. These eigenvalue problems are based on the Neumann matrix
associated to the non-overlapping subdomains $\Omega_i$, therefore, differ from those in
GenEO \cite{MR3093793} and AGDSW \cite{MR3812062}.

We also notice that in \cref{exact_eigen}, $u_{\Gamma_i}^TS^{(i)}u_{\Gamma_i}$ equal to the energy norm of the $a_i$-discrete harmonic extension of $u_{\Gamma_i}$ in $\Omega_i$, and $u_{\Gamma_i}^TA_{\Gamma\Gamma}^{(i)}u_{\Gamma_i}$ equal to the energy norm of zero extension of $u_{\Gamma_i}$ in $\Omega_i$.

In our spectral method, first we find an optimal space $Q^{(i)}$ (which will be defined below) so that in this optimal space, the resulting $a_0(\cdot,\cdot)$ is equivalent to the bilinear form of Schur complement $s(\cdot,\cdot)$, independent of the heterogeneity of the coefficients in $\Gamma_i^\delta$. This is the result of \Cref{theo1}, see below. Next we find a space $P^{(i)}$ to represent the best $k_i$-dimensional subspace of $V_i$, and use $P^{(i)}$ construct a local extension operator $R_0^{(i)^T}$ to approximate the $a_i$-discrete harmonic extension operator $\mathcal{H}^{(i)}$, see below.  

We start by fixing a threshold $\eta=O(\frac{h}{H})$, and choose the generalized eigenvalues in \cref{exact_eigen} for each subdomain, such
that $0\leq \lambda_1^{(i)}  \leq \cdots \leq \lambda_{k_i}^{(i)} < \eta=O(\frac{h}{H})\leq\lambda_{k_i+1}^{(i)} \leq \cdots \leq \lambda_{n_i}^{(i)} \leq 1 $. In \Cref{numerical} we will show how to determine exact value of $\eta$ and the number of eigenvalues smaller than $\eta$ in each subdomain. Now we just assume that $\eta=O(h/H)$ and there are $k_i$ eigenvalues smaller than $\eta$ in $\Omega_i.$

We choose the smallest $k_i$ eigenvalues and corresponding eigenvectors in $\cref{exact_eigen}$: for $ j=1,\cdots,k_i$, let $Q_j^{(i)}=\xi_j^{(i)}$, $P_j^{(i)}=-(A_{II}^{(i)})^{-1}A^{(i)}_{I \Gamma }\xi_j^{(i)}$, $Q^{(i)}=[Q^{(i)}_1,\cdots, Q_{k_i}^{(i)}]$, $P^{(i)}=[P^{(i)}_1,\cdots, P_{k_i}^{(i)}]$, and $D^{(i)}= \text{diagonal}(1-\lambda_1,\cdots,1-\lambda_{k_i})$. Then we have the following three identities, where the left-hand sides involve operators on $\Gamma_i$ only: 
\begin{enumerate}
    \item \hspace{15pt}$-A^{(i)}_{\Gamma \Gamma}Q^{(i)}D^{(i)} = A^{(i)}_{\Gamma I}P^{(i)}$,
    \item \hspace{15pt}$-D^{(i)}Q^{(i)^T}A^{(i)}_{\Gamma \Gamma}= P^{(i)^T}A^{(i)}_{I\Gamma}$,
    \item \hspace{15pt}$D^{(i)}Q^{(i)^T}A_{\Gamma \Gamma}^{(i)}Q^{(i)}=Q^{(i)^T}A_{\Gamma \Gamma}^{(i)}Q^{(i)}D^{(i)} =
      P^{(i)^T}A^{(i)}_{II}P^{(i)}$.
\end{enumerate}
The vectors $Q^{(i)}$ consist of the generalized eigenvectors from \cref{exact_eigen}. If we wish we can normalize $Q^{(i)}$ so
that $Q^{(i)^T}A^{(i)}_{\Gamma \Gamma}Q^{(i)} $ are identity matrices and 
$Q^{(i)^T}S^{(i)} Q^{(i)}$ are diagonal matrices with eigenvalues
on the diagonal. In the implementation and in the paper, we do not use normalized eigenvectors, so we keep $Q^{(i)^T}A^{(i)}_{\Gamma \Gamma}Q^{(i)}$.

The $Q^{(i)}$ is treated as the eigenfunction on $\Gamma_i$, and $P^{(i)}$ is the lower dimensional $a_i$-discrete harmonic extension from $\Gamma_i$ to the interior $\Omega_{i}$. 
Similarly we define the local extension operator $R_0^{(i)^T}:V_h(\Gamma_i)\to V_h(\Omega_i)$ as:
\begin{equation*}
    R_0^{(i)^T}u_{\Gamma_i}  :=\begin{cases}
\hspace{46pt}u_{\Gamma_i} \hspace{124pt} \mbox{on}~~ \Gamma_{ih},\\
 -P^{(i)}(P^{(i)^T}A_{II}^{(i)}P^{(i)})^{-1}P^{(i)^T}A_{I\Gamma}^{(i)}u_{\Gamma_i} \hspace{30pt} \mbox{in}~~\Omega_{ih},
\end{cases}
\end{equation*}
or in terms of $A^{(i)}_{\Gamma\Gamma}$ and $Q^{(i)}$:
\begin{equation*}
    R_0^{(i)^T}u_{\Gamma_i}  =\begin{cases}
\hspace{46pt}u_{\Gamma_i}  \hspace{118pt} \mbox{on}~~ \Gamma_{ih},\\
 P^{(i)}(Q^{(i)^T}A_{\Gamma\Gamma}^{(i)}Q^{(i)})^{-1}Q^{(i)^T}A_{\Gamma\Gamma}^{(i)}u_{\Gamma_i} \hspace{30pt} \mbox{in}~~\Omega_{ih}.
\end{cases}
\end{equation*}

Define the global extension ${R}_0^T:V_0\to V_h(\Omega)$ as:
\begin{small}
\begin{equation*}
    {R}_0^T\!u_\Gamma\!=\!\!\!\begin{bmatrix} u_\Gamma  \\ 
    \!-\!\displaystyle{\sum_{i=1}^N}R_{\!I_iI}^{T}P^{(i)}\!(P^{(i)^T}\!\!\!A_{II}^{(i)}P^{(i)})^{\!-1}\!P^{(i)^T}\!\!\!A^{(i)}_{I\Gamma}R_{\Gamma_{\!i}\Gamma}u_\Gamma\!\\
        \end{bmatrix}\!\!\!=\!\!\begin{bmatrix} u_\Gamma  \\ 
    \!\displaystyle{\sum_{i=1}^N} R_{\!I_iI}^{T}P^{(i)}\!(Q^{(i)^T}\!\!\!A^{(i)}_{\Gamma\Gamma}Q^{(i)})^{\!-1}Q^{(i)^T}\!\!\!A_{\Gamma\Gamma}^{(i)}R_{\Gamma_{\!i}\Gamma}u_\Gamma\!\\
        \end{bmatrix}\!\!.
\end{equation*}
\end{small}
Notice that the second part of the equation can be simplified as:
\begin{equation*}
    -\displaystyle{\sum_{i=1}^N}R_{I_iI}^{T}P^{(i)}(P^{(i)^T}\!\!A_{II}^{(i)}P^{(i)})^{-1}P^{(i)^T}\!\!A^{(i)}_{I\Gamma}R_{\Gamma_i\Gamma}=-P(P^T\!A_{II}P)^{-1}P^T\!A_{I\Gamma},
\end{equation*}
where 
\begin{equation*}
    A_{II}=\sum_{i=1}^NR_{I_iI}^{T}A_{II}^{(i)}R_{I_iI}\hspace{10pt}\text{and}\hspace{20pt}A_{I\Gamma}=\sum_{i=1}^NR_{I_iI}^{T}A_{I\Gamma}^{(i)}R_{\Gamma_i\Gamma},
\end{equation*}
and $P: \vec{u}  \to \displaystyle{\bigcup_{i=1}^N} V_i $  defined as:
\begin{equation*}
    P=\sum_{i=1}^NR_{I_iI}^{T}P^{(i)}R_{\lambda_i}.
\end{equation*}
Here $R_{\lambda_i}$ is a restriction choosing $[{u}_{i1}\!,\!\cdots\!,\!{u}_{ik_i}]^T$ from $\vec{u}=[{u}_{11}\!,\!\cdots\!,\!{u}_{1k_1}\!,\!\cdots\!,\!{u}_{Nk_1}\!,\!\cdots\!,\!{u}_{Nk_N}\!]^T$\!\!, $k_i$ is the number of eigenfunctions chosen from the i-th subdomain, and $\vec{u}$ have dimension  $N_E=\!\!\!\!\displaystyle{\sum_{1\leq i\leq N}\!\!\!k_i}$, the number of all eigenvectors we chosen from all N subdomains.

And we define the exact coarse bilinear form as:
\begin{equation*}
    \begin{split}
    {a}_0(u_\Gamma,v_\Gamma)=a({R}_0^Tu_\Gamma,{R}_0^Tv_\Gamma)&=v_\Gamma^T\sum_{i=1}^NR_{\Gamma_i\Gamma}^{T}\big(A^{(i)}_{\Gamma \Gamma }\!-\!A^{(i)}_{\Gamma I}P^{(i)}(P^{(i)^T}A^{(i)}_{II}P^{(i)})^{-1}P^{(i)^T}\!\!A^{(i)}_{I\Gamma}\big)R_{\Gamma_i\Gamma}u_\Gamma\\
    &=v_\Gamma^T(A_{\Gamma \Gamma }-A_{\Gamma I }P(P^TA_{II }P)^{-1}P^TA_{I \Gamma })u_\Gamma,   \hspace{40pt} \forall u_\Gamma,v_\Gamma\in V_0,
\end{split}
\end{equation*}
or in terms of $A^{(i)}_{\Gamma\Gamma}$ and $Q^{(i)}$:
\begin{equation*}
    {a}_0(u_\Gamma,v_\Gamma)\!\!=v_\Gamma^T\sum_{i=1}^NR_{\Gamma_i\Gamma}^{T}\big(A^{(i)}_{\Gamma \Gamma }\!-\!A^{(i)}_{\Gamma\Gamma}Q^{(i)}D^{(i)}(Q^{(i)^T}\!\!A^{(i)}_{\Gamma\Gamma}Q^{(i)})^{-1}Q^{(i)^T}\!\!A^{(i)}_{\Gamma\Gamma}\big)R_{\Gamma_i\Gamma}u_\Gamma, \hspace{15pt} \forall u_\Gamma,v_\Gamma\in V_0.
\end{equation*}
The coarse solution $w_\Gamma = \tilde{T}_0 u_h$ by the coarse problem \cref{localpb2} can be obtained in matrix form by:
\begin{equation*}
\small
   \!\!\sum_{i=1}^N\!\!  R_{\Gamma_{\!i}\Gamma}^{T}\!\big(\! A^{(i)}_{\Gamma\Gamma}\!-\!A^{(i)}_{\Gamma I}P^{(i)}\!(\!P^{(i)^T}\!\!\!A^{(i)}_{II}P^{(i)})^{\!-1}\!P^{(i)^T}\!\!\!A^{(i)}_{I\Gamma}\big)\!R_{\Gamma_{\!i}\Gamma}  w_\Gamma\!\!=\!\! \sum_{i=1}^N\!\! R_{\Gamma_{\!i}\Gamma}^{T}\!\big(  b^{(i)}_\Gamma\!\!-\!A^{(i)}_{\Gamma I}P^{(i)}\!(\!P^{(i)^T}\!\!\!A^{(i)}_{II}\!P^{(i)}\!)^{\!-1}\!P^{(i)^T}\!b_I^{(i)}\!\big),
\end{equation*}
or equivalently via:
\begin{equation*}
\small
   \!\!\sum_{i=1}^N\!\!R_{\Gamma_{\!i}\Gamma}^T\!\big(\!A^{(i)}_{\Gamma \Gamma }\!-\!A^{(i)}_{\Gamma\Gamma}\!Q^{(i)}\!\!D^{(i)}\!(\!Q^{(i)^T}\!\!\!\!A^{(i)}_{\Gamma\Gamma}Q^{(i)}\!)^{\!-1}\!Q^{(i)^T}\!\!\!A^{(i)}_{\Gamma\Gamma}\!\big)\!R_{\Gamma_{\!i}\Gamma}w_\Gamma\!\!=\!\!\!\sum_{i=1}^N\!\!R_{\Gamma_{\!i}\Gamma}^T\!\big(\! b_\Gamma^{(i)}\!+\!A^{(i)}_{\Gamma \Gamma}\!Q^{(i)}\!(\!Q^{(i)^T}\!\!\!A^{(i)}_{\Gamma \Gamma}\!Q^{(i)}\!)^{\!-1}\!P^{(i)^T}\!b_I^{(i)}\!\big).
\end{equation*}

Now we will show the condition number of NOSAS is only associated with the eigenvalues of \cref{exact_eigen} greater than $\eta$. First, let us prove a theorem that holds in each subdomain locally.
\begin{theorem} \label{theo1} 
  Let $\Pi^{(i)}_Su_{\Gamma_i}$ be the projection of $u_{\Gamma_i}\in V_h(\Gamma_i)$ onto the eigenfunctions space, Span$\{Q^{(i)}\}$. That
  is, $\Pi^{(i)}_Su_{\Gamma_i}:= Q^{(i)}(Q^{(i)^T}\!\!A_{\Gamma\Gamma}^{(i)}Q^{(i)})^{-1}Q^{(i)^T}\!\!A_{\Gamma\Gamma}^{(i)}u_{\Gamma_i}$. Let define the local bilinear form for $i=1,\cdots, N$:
  \begin{equation*}
      a^{(i)}_0(u_{\Gamma_i},v_{\Gamma_i})= v_{\Gamma_i}^T(A^{(i)}_{\Gamma \Gamma }\!-\!A^{(i)}_{\Gamma\Gamma}Q^{(i)}D^{(i)}(Q^{(i)^T}\!\!\!A^{(i)}_{\Gamma\Gamma}Q^{(i)})^{-1}Q^{(i)^T}\!\!\!A^{(i)}_{\Gamma\Gamma}\big)u_{\Gamma_i}, 
  \end{equation*}
  where $u_{\Gamma_i},v_{\Gamma_i}\in V_h(\Gamma_i)$. Then, we have
  \[
    {a}^{(i)}_0(u_{\Gamma_i},v_{\Gamma_i})= (\Pi_S^{(i)}v_{\Gamma_i})^T S^{(i)} (\Pi_S^{(i)}u_{\Gamma_i})+ (v_{\Gamma_i} - \Pi_S^{(i)}v_{\Gamma_i})^T
    A^{(i)}_{\Gamma \Gamma} (u_{\Gamma_i}- \Pi_S^{(i)}u_{\Gamma_i}).
  \]
\end{theorem} 
\begin{proof}

Let us denote $u_{\Gamma_i}=u_1+u_2$, $v_{\Gamma_i}=v_1+v_2$, $u_{\Gamma_i},v_{\Gamma_i}\in V_h(\Gamma_i)$. And $u_1 = \Pi^{(i)}_Su_{\Gamma_i}$, $u_2 = u_{\Gamma_i} -\Pi^{(i)}_Su_{\Gamma_i}$, $v_1 = \Pi^{(i)}_Sv_{\Gamma_i}$, $v_2 = v_{\Gamma_i} -\Pi^{(i)}_Sv_{\Gamma_i}$. Then  
    \begin{equation*}
    {a}^{(i)}_0(u_{\Gamma_i},v_{\Gamma_i})=(v_1+v_2)^T(A^{(i)}_{\Gamma \Gamma }-A^{(i)}_{\Gamma \Gamma }Q^{(i)}D^{(i)}(Q^{(i)^T}A^{(i)}_{\Gamma \Gamma }Q^{(i)})^{-1}Q^{(i)^T}A^{(i)}_{\Gamma \Gamma })(u_1+u_2).
\end{equation*}
We note that $v_1^TA^{(i)}_{\Gamma \Gamma }u_2=0$, $v_2^TA^{(i)}_{\Gamma \Gamma }u_1=0$ and $Q^{(i)^T}A^{(i)}_{\Gamma \Gamma }u_2=0$. Thus
\begin{equation*}
    {a}^{(i)}_0(u_{\Gamma_i},v_{\Gamma_i})=v_1^T(A^{(i)}_{\Gamma \Gamma }-A^{(i)}_{\Gamma \Gamma }Q^{(i)}D^{(i)}(Q^{(i)^T}A^{(i)}_{\Gamma \Gamma }Q^{(i)})^{-1}Q^{(i)^T}A^{(i)}_{\Gamma \Gamma })u_1+v_2^TA_{\Gamma\Gamma}^{(i)}u_2.
\end{equation*}
For any $\xi \in\text{Span}(Q^{(i)})$,
$$\xi^TA^{(i)}_{\Gamma \Gamma }Q^{(i)}D^{(i)}(Q^{(i)^T}A^{(i)}_{\Gamma \Gamma }Q^{(i)})^{-1}Q^{(i)^T}A^{(i)}_{\Gamma \Gamma }\xi=(1-\lambda^{(i)})\xi^TA^{(i)}_{\Gamma \Gamma}\xi,$$ so
$$\xi^T(A^{(i)}_{\Gamma \Gamma }-A^{(i)}_{\Gamma \Gamma }Q^{(i)}D^{(i)}(Q^{(i)^T}A^{(i)}_{\Gamma \Gamma }Q^{(i)})^{-1}Q^{(i)^T}A^{(i)}_{\Gamma \Gamma })\xi=\lambda \xi^TA^{(i)}_{\Gamma \Gamma}\xi=\xi^TS^{(i)}\xi,$$ 
and therefore 
$$v_1^T(A^{(i)}_{\Gamma \Gamma }-A^{(i)}_{\Gamma \Gamma }Q^{(i)}D^{(i)}(Q^{(i)^T}A^{(i)}_{\Gamma \Gamma }Q^{(i)})^{-1}Q^{(i)^T}A^{(i)}_{\Gamma \Gamma })u_1=v_1^TS^{(i)}u_1.$$ Hence,
\begin{equation*}
     {a}^{(i)}_0(u_{\Gamma_i},v_{\Gamma_i})=v_1^TS^{(i)}u_1+v_2^TA^{(i)}_{\Gamma \Gamma }u_2.
\end{equation*}
\end{proof}

We will also need the analysis of the following results.

\begin{lemma}\label{bounda0} 
Let $\eta$ be the threshold of \cref{exact_eigen}, for all $u_\Gamma \in V_0$ holds,
  \[
  {a}_0(u_\Gamma,u_\Gamma) =\sum_{i=1}^Na^{(i)}_0(R_{\Gamma_i\Gamma}u_\Gamma,R_{\Gamma_i\Gamma}u_\Gamma)\leq \sum_{i=1}^N\frac{1}{\eta}u_\Gamma^TR_{\Gamma_i\Gamma}^{T} S^{(i)} R_{\Gamma_i\Gamma}u_\Gamma=\frac{1}{\eta}u_\Gamma^T S u_\Gamma.
  \]
 \end{lemma} 
\begin{proof} 
First remember the property of generalized eigenproblem \cref{exact_eigen}, 
$$v_{\Gamma_i}^TS^{(i)}v_{\Gamma_i} <  \eta\; v_{\Gamma_i}^T A^{(i)}_{\Gamma\Gamma}v_{\Gamma_i}  \quad \forall v_{\Gamma_i} \in \mbox{Range}(Q^{(i)}),$$
and 
  $$v_{\Gamma_i}^TS^{(i)}v_{\Gamma_i} \ge  \eta\; v_{\Gamma_i}^T A^{(i)}_{\Gamma\Gamma}v_{\Gamma_i} \quad \forall v_{\Gamma_i} \in \mbox{Range}(Q^{(i)^{\perp}}).$$
Then using \cref{theo1} and denoting $u^{(i)} = \Pi^{(i)}_S R_{\Gamma_i\Gamma}u_\Gamma$ and $v^{(i)} = R_{\Gamma_i\Gamma}u_\Gamma - u^{(i)}$, we have 
\begin{equation*}
    \begin{split}
{a}^{(i)}_0(R_{\Gamma_i\Gamma}u_\Gamma,R_{\Gamma_i\Gamma}u_\Gamma) = & u^{(i)^T}\!\!S^{(i)}u^{(i)}+ v^{(i)^T}\!\!A^{(i)}_{\Gamma \Gamma} v^{(i)} \leq u^{(i)^T}\!\!S^{(i)}u^{(i)}+\frac{1}{\eta}v^{(i)^T}S^{(i)}v^{(i)} \\
\leq& \frac{1}{\eta}  u^{(i)^T}\!\!S^{(i)}u^{(i)}+\frac{1}{\eta}v^{(i)^T}S^{(i)}v^{(i)}   \hspace{20pt}\text{(use the fact that}\hspace{5pt} 0<\eta\leq1)\\
=&\frac{1}{\eta}  u_\Gamma^TR_{\Gamma_i\Gamma}^{T}S^{(i)}R_{\Gamma_i\Gamma}u_\Gamma\hspace{20pt}\text{(this is true for all} \hspace{5pt}1\leq i \leq N). 
\end{split}
\end{equation*}
Also, from the definition of ${a}_0(u_\Gamma,u_\Gamma)$ and ${a}^{(i)}_0(u_{\Gamma_i},u_{\Gamma_i})$, we have 
\begin{equation*}
\begin{split}
  {a}_0(u_\Gamma,u_\Gamma)\!\!&=u_\Gamma^T\!\sum_{i=1}^NR_{\Gamma_i\Gamma}^{T}\big(A^{(i)}_{\Gamma \Gamma }\!-\!A^{(i)}_{\Gamma\Gamma}Q^{(i)}D^{(i)}(Q^{(i)^T}\!\!A^{(i)}_{\Gamma\Gamma}Q^{(i)})^{-1}Q^{(i)^T}\!\!A^{(i)}_{\Gamma\Gamma}\big)R_{\Gamma_i\Gamma}u_\Gamma\\
  &=\sum_{i=1}^Na_0^{(i)}(R_{\Gamma_i\Gamma}u_\Gamma,R_{\Gamma_i\Gamma}u_\Gamma)\leq \sum_{i=1}^N\frac{1}{\eta}u_\Gamma^TR_{\Gamma_i\Gamma}^{T} S^{(i)} R_{\Gamma_i\Gamma}u_\Gamma=\frac{1}{\eta}u_\Gamma^T S u_\Gamma.
  \end{split}
\end{equation*}
\end{proof}
Notice the above \cref{bounda0} still holds if replace $\eta$ by $\lambda_{k_i+1}^{(i)}$, that is the smallest eigenvalue which greater than $\eta$ in \cref{exact_eigen}. Denoted  $\lambda_{min}(\eta)=\displaystyle{\min_{1\leq i\leq N}}\{\lambda_{k_i+1}^{(i)} \}$, the smallest eigenvalue which greater than $\eta$ for all subdomain. In view of the abstract theory of ASM, see \cite[Chapter 2]{MR2104179}, the following three key assumptions can lead to the condition number of NOSAS.

\begin{lemma} (Assumption i) Let $C_0^2 = 2 + \frac{3}{\lambda_{min}(\eta)}$. Then,
  for any $u\in V_h(\Omega)$, there exist $u_i\in V_i$ for  $0\leq i\leq N$, such that $u={R}_0^Tu_0+ \sum_{i=1}^N R_{i}^{T}u_i$ and satisfies
  \[ \sum_{i=0}^N {a}_i(u_i,u_i)\leq  C_0^2 a(u,u).\]
\end{lemma}  
  \begin{proof} The decomposition is unique, given by $u_0 = u_\Gamma$ and
   the others $u_i$ obtained from
    $\displaystyle{\sum_{i=1}^N  R_{i}^{T}u_i = u - R_0^T  u_\Gamma}$. Hence, our decomposition
   satisfies \\
    \begin{equation*}
        \begin{split}
          \sum_{i=0}^N{a}_i(u_i,u_i)
        &={a}_0(u_0,u_0)+\sum_{i=1}^N{a}_i(u_i,u_i)\\
        &=a({R}_0^Tu_0,{R}_0^Tu_0)+\sum_{i=1}^Na(R^{T}_{i}u_i,R^{T}_{i}u_i)\\
        &=a({R}_0^Tu_0,{R}_0^Tu_0)+a(u-{R}_0^Tu_0,u-{R}_0^Tu_0) \hspace{10pt}\text{(orthogonality of each subdomain)}\\
        &\leq a({R}_0^Tu_0,{R}_0^Tu_0)+2a(u,u)+2a({R}_0^Tu_0,{R}_0^Tu_0)\\
        &\leq 2a(u,u)+\frac{3}{\lambda_{min}(\eta)}a(\mathcal{H}u_0,\mathcal{H}u_0)\leq(2+\frac{3}{\lambda_{min}(\eta)})a(u,u) \hspace{10pt}\text{(use  \Cref{bounda0}}).
        \end{split}
    \end{equation*}
\end{proof}

\begin{lemma} (Assumption ii) We have $\mu(\epsilon)=1$ for the spectral radius of matrix $\epsilon=\{\epsilon_{ij}\}_{i,j=1,\cdots,N}$, defined by
$$|a(R_i^Tu_i,R_j^Tu_j)|\leq \epsilon_{ij}a^{1/2}(R_i^Tu_i,R_i^Tu_i)a^{1/2}(R_j^Tu_j,R_j^Tu_j)  \hspace{20pt} \forall u_i\in V_i,\text{ and   }\forall u_j\in V_j.$$
\end{lemma}
\begin{proof} 
  In our method, $V_i$ and $V_j$ are orthogonal for $i,j=1,\cdots,N$ and $i\not=j$, therefore $\mu(\epsilon)=1$.
\end{proof}

 \begin{lemma} (Assumption iii) We have 
    $$a(R_i^Tu_i,R_i^Tu_i)\leq \omega_i {a}_i(u_i,u_i) \hspace{20pt} \forall u_i\in V_i, 0\leq i\leq N.
    $$
\end{lemma}
\begin{proof}
  We have equality with $\omega_i = 1$  for $0\leq i \leq N$, from
  the definition of the $a_i(\cdot,\cdot)$. 
\end{proof} 

\begin{theorem} \label{theo2} For any $u\in V_h(\Omega)$, the following holds:
  $$ (2+\frac{3}{\lambda_{min}(\eta)})^{-1} a(u,u)\leq a(T_Au,u) \leq 2a(u,u),$$
where $T_A$  was defined in \cref{Def_T_A} and $\lambda_{min}(\eta)$ defined in previous page.
 \end{theorem}
\begin{proof}
Using the general framework of additive Schwarz methods.  The lower bound is given by $C_0^{-2}$ and the upper bound by
  $\displaystyle{\max_{1\leq i\leq N}} \{\omega_i\} (\mu(\epsilon) + 1)$. For details see \cite[Chapter 2]{MR2104179}.
\end{proof}

 \Cref{theo2} shows an estimate for the condition number of NOSAS. We can choose $\eta=O(h/H)$ to guarantee the condition number is $O(H/h)$.  We also note that if we choose one eigenvalue in each subdomain, the
  NOSAS is better than MES. On the floating subdomain, both methods
  are equivalent by choosing the best constant function extension
  in the interior nodes of the subdomain. On the subdomain that touches
  the Dirichlet boundary $\partial \Omega$, those two methods differ, while
  in MES selects the best constant function extension in the interior nodes,
  NOSAS find the best one-dimensional function extension by solving
  a generalized eigenvalue problem.

\section{COMPLEXITY OF THE COARSE PROBLEM AND NOSAS WITH INEXACT SOLVER}\label{Section4}
\subsection{Implementation and complexity of NOSAS}\hspace*{\fill} 

The solution $w_\Gamma = \tilde{T}_0 u_h$ of the coarse problem
\begin{equation*}
    {a}_0(\tilde{T}_0u_h,v_{\Gamma})=a(u_h,{R}_0^Tv_{\Gamma}) = ({R}_0^Tv_{\Gamma})^T b \quad\quad \forall v_\Gamma\in V_0,
\end{equation*} 
is of the form:
\begin{equation*}
\small
   \!\!\sum_{i=1}^N\!\!R_{\Gamma_{\!i}\Gamma}^T\!\big(\!A^{(i)}_{\Gamma \Gamma }\!-\!A^{(i)}_{\Gamma\Gamma}\!Q^{(i)}\!\!D^{(i)}\!(\!Q^{(i)^T}\!\!\!\!A^{(i)}_{\Gamma\Gamma}Q^{(i)}\!)^{\!-1}\!Q^{(i)^T}\!\!\!A^{(i)}_{\Gamma\Gamma}\!\big)\!R_{\Gamma_{\!i}\Gamma}w_\Gamma\!\!=\!\!\!\sum_{i=1}^N\!\!R_{\Gamma_{\!i}\Gamma}^T\!\big(\! b_\Gamma^{(i)}\!+\!A^{(i)}_{\Gamma \Gamma}\!Q^{(i)}\!(\!Q^{(i)^T}\!\!\!A^{(i)}_{\Gamma \Gamma}\!Q^{(i)}\!)^{\!-1}\!P^{(i)^T}\!b_I^{(i)}\!\big).
\end{equation*}
By summing all local matrices, let $\displaystyle{
    \sum_{i=1}^N}R_{\Gamma_i\Gamma}^{T}A^{(i)}_{\Gamma\Gamma }R_{\Gamma_i\Gamma}\!=\!A_{\Gamma\Gamma }$, $\displaystyle{\sum_{i=1}^N}R^T_{\Gamma_i\Gamma}A^{(i)}_{\Gamma\Gamma }Q^{(i)}R_{\lambda_i}\!=\!U$,\\
 $\displaystyle{\sum_{i=1}^N}R_{\lambda_i}^{T}D^{(i)}R_{\lambda_i}\!=\!D$,\quad$\displaystyle{\sum_{i=1}^N}R_{\lambda_i}^{T}(Q^{(i)^T}\!A^{(i)}_{\Gamma\Gamma}Q^{(i)})^{-1}R_{\lambda_i}\!=\!C$,\quad and $\displaystyle{P=\sum_{i=1}^N}R_{I_iI}^{T}P^{(i)}R_{\lambda_i}$.\\
Here $R_{\lambda_i}$ is the restriction choosing $[{u}_{i1}\!,\!\cdots\!,\!{u}_{ik_i}]^T$ from $\vec{u}=[{u}_{11}\!,\!\cdots\!,\!{u}_{1k_1}\!,\!\cdots\!,\!{u}_{Nk_1}\!,\!\cdots\!,\!{u}_{Nk_N}\!]^T$\!\!, $k_i$ is the number of eigenfunctions chosen from the i-th subdomain, and $\vec{u}$ have dimension  $N_E$, the number of all eigenvectors we chosen from all N subdomains.

Then we can rewrite the coarse problem into global matrices: 
\begin{equation*}
    (A_{\Gamma\Gamma }-UDCU^T)w_\Gamma=b_\Gamma+UCP^Tb_I,
\end{equation*}
and we use Woodbury matrix identity for implementation:
\begin{equation*}
    (A_{\Gamma\Gamma }-UDCU^T)^{-1}=A_{\Gamma\Gamma }^{-1}+A_{\Gamma\Gamma }^{-1}U(C^{-1}D^{-1}-U^TA_{\Gamma\Gamma }^{-1}U)^{-1}U^TA_{\Gamma\Gamma }^{-1}.
\end{equation*}
 Note that $C^{-1}$, $D^{-1}$ are diagonal matrices, then the complexity of the method is associated with  $A_{\Gamma\Gamma }^{-1}$ and the $N_E \times N_E$ matrix $(C^{-1}D^{-1}-U^TA_{\Gamma\Gamma }^{-1}U)^{-1}$.

The motivation to simplify the coarse problem is to make 
 $A_{\Gamma\Gamma }$ to be block diagonal or diagonal matrix, so the only global component of the coarse problem is only associated with the $N_E \times N_E$ matrix $(C^{-1}D^{-1}-U^TA_{\Gamma\Gamma }^{-1}U)^{-1}$. If in each subdomain, we replace the exact ${A}^{(i)}_{\Gamma\Gamma}$ on the right-hand side of the generalized eigenvalue problem by $\hat{A}^{(i)}_{\Gamma\Gamma}$, where $\hat{A}^{(i)}_{\Gamma\Gamma}$ is the block diagonal or diagonal version of ${A}^{(i)}_{\Gamma\Gamma}$, the global assembling matrix  $\hat{A}_{\Gamma\Gamma }$ will be block diagonal or diagonal respectively. For the block diagonal case, we eliminate the connections across different faces, edges, and corners of the subdomain. For the diagonal case, we eliminate the connections across different vertices. These inexact cases can be analyzed and given in the following subsection.

\subsection{New method: non-overlapping spectral Schwarz method with inexact solver(NOSAS)}\hspace*{\fill} 

In the inexact coarse solver, the idea and the definitions are similar to the exact solver; we use the same decomposition and the same local solver as before. The only change is in each subdomain; we introduce the following local generalized eigenvalue problem with the block diagonal or diagonal version of ${A}^{(i)}_{\Gamma\Gamma}$ as $\hat{A}^{(i)}_{\Gamma\Gamma}$. The block diagonal of ${A}^{(i)}_{\Gamma\Gamma}$ obtained by eliminating the value between different faces, edges, and corners of the subdomain. 
The diagonal of ${A}^{(i)}_{\Gamma\Gamma}$ is obtained by eliminating the value between different vertices:
\begin{equation}\label{eigeninexact} 
 S^{(i)} \hat{\xi}^{(i)}_j := (A^{(i)}_{\Gamma\Gamma}-A^{(i)}_{\Gamma I}(A^{(i)}_{II})^{-1}A^{(i)}_{I \Gamma })\hat{\xi}_j^{(i)}=\hat{\lambda}_j^{(i)} \hat{A}^{(i)}_{\Gamma\Gamma}\hat{\xi}_j^{(i)},
\hspace{30pt} (j=1,\cdots, n_i)\end{equation}
where $n_i$ is the degrees of freedom on $\Gamma_i$, and $0\leq\hat{\lambda}_1^{(i)} \leq \cdots \leq \hat{\lambda}_{n_i}^{(i)} $. We note that for this case we do not have necessarily $\hat{\lambda}_j^{(i)} \leq 1$ for all $j$ and $i$.  We choose the smallest $k_i$ eigenvalues which are less than the threshold $\eta=O(h/H)$ and denote $\hat{Q}^{(i)}=[\hat{\xi}^{(i)}_1,\hat{\xi}^{(i)}_2,\cdots, \hat{\xi}_{k_i}^{(i)}]$,  $\hat{P}^{(i)}=-(A_{II}^{(i)})^{-1}A^{(i)}_{I \Gamma }\hat{Q}^{(i)}$, and similar as before, $
\hat{D}^{(i)}= \text{diagonal}(1-\hat{\lambda}^{(i)}_1,1-\hat{\lambda}^{(i)}_2,\cdots,1-\hat{\lambda}^{(i)}_{k_i})$.
 
 Define the global extension $\hat{R}_0^T:V_0\to V_h(\Omega)$ as:
\begin{equation*}
    \hat{R}_0^Tu_\Gamma\!:=\!\!\begin{bmatrix} u_\Gamma  \\ 
    \displaystyle{\sum_{i=1}^N} R_{I_iI}^{T}\hat{P}^{(i)}\!(\hat{Q}^{(i)^T}\!\!\hat{A}^{(i)}_{\Gamma\Gamma}\hat{Q}^{(i)})^{-1}\hat{Q}^{(i)^T}\!\!\hat{A}_{\Gamma\Gamma}^{(i)}R_{\Gamma_i\Gamma}u_\Gamma\\
        \end{bmatrix}.
\end{equation*}

Next we define the inexact coarse bilinear form as:
\begin{equation*}
    \hat{a}_0(u_{\Gamma},v_{\Gamma})=v_{\Gamma}^T\!\sum_{i=1}^NR_{\Gamma_i\Gamma}^{T}\big(\hat{A}^{(i)}_{\Gamma \Gamma }\!-\!\hat{A}^{(i)}_{\Gamma\Gamma}\hat{Q}^{(i)}\hat{D}^{(i)}(\hat{Q}^{(i)^T}\!\!\hat{A}^{(i)}_{\Gamma\Gamma}\hat{Q}^{(i)})^{-1}\hat{Q}^{(i)^T}\!\!\hat{A}^{(i)}_{\Gamma\Gamma}\big)R_{\Gamma_i\Gamma}u_{\Gamma} \hspace{10pt} \forall u_{\Gamma},v_{\Gamma}\in V_0.
\end{equation*}

We note $\hat{a}_0(u_{\Gamma},v_{\Gamma}) \neq a(\hat{R}_0^Tu_{\Gamma}, \hat{R}_0^Tv_{\Gamma})$ for $u_{\Gamma},v_{\Gamma} \in V_0$; see \Cref{theo3} and  \Cref{bound1} below. Similarly, the coarse solution  $w_\Gamma = \hat{T}_0 u_h$  by the coarse problem \cref{localpb2} can be obtained in matrix form by:
\begin{equation*}
\small
   \sum_{i=1}^N\!\!R_{\Gamma_{\!i}\Gamma}^{T}\!\big(\!\hat{A}^{(i)}_{\Gamma \Gamma }\!-\!\hat{A}^{(i)}_{\Gamma\Gamma}\hat{Q}^{(i)}\!\!\hat{D}^{(i)}\!(\hat{Q}^{\!(i)^{\!T}}\!\!\!\hat{A}^{(i)}_{\Gamma\Gamma}\hat{Q}^{(i)}\!)^{\!-1}\!\hat{Q}^{\!(i)^{\!T}}\!\!\!\hat{A}^{(i)}_{\Gamma\Gamma}\!\big)\!R_{\Gamma_{\!i}\Gamma}\!w_\Gamma\!\!=\!\!\!\sum_{i=1}^N\!\!R_{\Gamma_{\!i}\Gamma}^{T}\!\big(\! b_\Gamma^{(i)}\!\!+\!\hat{A}^{(i)}_{\Gamma \Gamma}\hat{Q}^{(i)}\!(\!\hat{Q}^{\!(i)^{\!T}}\!\!\hat{A}^{(i)}_{\Gamma \Gamma}\hat{Q}^{\!(i)})^{\!-1}\!\hat{P}^{(i)^T}\!\!b_I^{(i)}\!\big)\!.
\end{equation*}
Also, if we consider the local property, we will have similar theorem below:
\begin{theorem} \label{theo3} 
  Let $\hat{\Pi}^{(i)}_Su_{\Gamma_i}$ be the projection of $u_{\Gamma_i}\in V_h(\Gamma_i)$ onto the eigenfunctions space , Span$\{\hat{Q}^{(i)}\}$. That
  is, $\hat{\Pi}^{(i)}_Su_{\Gamma_i}:= \hat{Q}^{(i)}(\hat{Q}^{(i)^T}\hat{A}^{(i)}_{\Gamma\Gamma}\hat{Q}^{(i)})^{-1}\hat{Q}^{(i)^{T}}\hat{A}^{(i)}_{\Gamma\Gamma}u_{\Gamma_i}$. Let define the local bilinear form for $i=1,\cdots,N$:
  \begin{equation*}
      \hat{a}^{(i)}_0(u_{\Gamma_i},v_{\Gamma_i})= v_{\Gamma_i}^T(\hat{A}^{(i)}_{\Gamma \Gamma }\!-\!\hat{A}^{(i)}_{\Gamma\Gamma}\hat{Q}^{(i)}\hat{D}^{(i)}(\hat{Q}^{(i)^T}\!\!\hat{A}^{(i)}_{\Gamma\Gamma}\hat{Q}^{(i)})^{-1}\hat{Q}^{(i)^T}\!\!\hat{A}^{(i)}_{\Gamma\Gamma}\big)u_{\Gamma_i},
  \end{equation*} 
  where $u_{\Gamma_i},v_{\Gamma_i}\in V_h(\Gamma_i)$. Then, we have
\begin{equation*}
    \hat{a}^{(i)}_0(u_{\Gamma_i},v_{\Gamma_i})= (\hat{\Pi}_S^{(i)}v_{\Gamma_i})^T S^{(i)} (\hat{\Pi}_S^{(i)}u_{\Gamma_i})+ (v_{\Gamma_i} - \hat{\Pi}_S^{(i)}v_{\Gamma_i})^T
    \hat{A}^{(i)}_{\Gamma \Gamma} (u_{\Gamma_i}- \hat{\Pi}_S^{(i)}u_{\Gamma_i}).
    \end{equation*}
\end{theorem}

\begin{proof}
Let us denote $u_{\Gamma_i}=u_1+u_2$, $v_{\Gamma_i}=v_1+v_2$, $u_{\Gamma_i},v_{\Gamma_i}\in V_h(\Gamma_i)$. And $u_1 = \hat{\Pi}^{(i)}_Su_{\Gamma_i}$, $u_2 = u_{\Gamma_i} -\hat{\Pi}^{(i)}_Su_{\Gamma_i}$, $v_1 = \hat{\Pi}^{(i)}_Sv_{\Gamma_i}$, $v_2 = v_{\Gamma_i} -\hat{\Pi}^{(i)}_Sv_{\Gamma_i}$. Then  
    \begin{equation*}
    \hat{a}^{(i)}_0(u_{\Gamma_i},v_{\Gamma_i})=(v_1+v_2)^T(\hat{A}^{(i)}_{\Gamma \Gamma }- \hat{A}^{(i)}_{\Gamma\Gamma}\hat{Q}^{(i)}\hat{D}^{(i)}(\hat{Q}^{(i)^T}\hat{A}^{(i)}_{\Gamma \Gamma }\hat{Q}^{(i)})^{-1}\hat{Q}^{(i)^T}\hat{A}^{(i)}_{\Gamma \Gamma })(u_1+u_2).
\end{equation*}
    We note that $v_1^T\hat{A}^{(i)}_{\Gamma \Gamma }u_2=0$, $v_2^T\hat{A}^{(i)}_{\Gamma \Gamma }u_1=0$ and $\hat{Q}^{(i)^T}\hat{A}^{(i)}_{\Gamma \Gamma }u_2=0$. Thus,
\begin{equation*}
    \hat{a}^{(i)}_0     (u,v)=v_1^T(\hat{A}^{(i)}_{\Gamma \Gamma }-\hat{D}^{(i)}\hat{A}^{(i)}_{\Gamma \Gamma }\hat{Q}^{(i)}(\hat{Q}^{(i)^T}\hat{A}^{(i)}_{\Gamma \Gamma }\hat{Q}^{(i)})^{-1}\hat{Q}^{(i)^T}\hat{A}^{(i)}_{\Gamma \Gamma })u_1+v_2^T\hat{A}^{(i)}_{\Gamma \Gamma }u_2.
\end{equation*}
For any $\hat{\xi} \in\text{Span}(\hat{Q^{(i)}})$,
$$\hat{\xi}^T\hat{A}^{(i)}_{\Gamma \Gamma }\hat{Q}^{(i)}\hat{D}^{(i)}(\hat{Q}^{(i)^T}\hat{A}^{(i)}_{\Gamma \Gamma }\hat{Q}^{(i)})^{-1}\hat{Q}^{(i)^T}\hat{A}^{(i)}_{\Gamma \Gamma }\hat{\xi}=(1-\hat{\lambda}_i)\hat{\xi}^T\hat{A}^{(i)}_{\Gamma \Gamma}\hat{\xi},$$ and 
$$\hat{\xi}^T(\hat{A}^{(i)}_{\Gamma \Gamma }-\hat{A}^{(i)}_{\Gamma \Gamma }\hat{Q}^{(i)}\hat{D}^{(i)}(\hat{Q}^{(i)^T}\hat{A}^{(i)}_{\Gamma \Gamma }\hat{Q}^{(i)})^{-1}\hat{Q}^{(i)^T}\hat{A}^{(i)}_{\Gamma \Gamma })\hat{\xi}=\hat{\lambda}_i\hat{\xi}^T\hat{A}^{(i)}_{\Gamma \Gamma}\hat{\xi}=\hat{\xi}^TS^{(i)}\hat{\xi},$$ and therefore, 
$$v_1^T(\hat{A}^{(i)}_{\Gamma \Gamma }-\hat{A}^{(i)}_{\Gamma \Gamma }\hat{Q}^{(i)}\hat{D}^{(i)}(\hat{Q}^{(i)^T}\hat{A}^{(i)}_{\Gamma \Gamma }\hat{Q}^{(i)})^{-1}\hat{Q}^{(i)^T}\hat{A}^{(i)}_{\Gamma \Gamma })u_1=v_1^TSu_1.$$
Hence,  
\begin{equation*}
  \hat{a}^{(i)}_0(u_{\Gamma_i},v_{\Gamma_i}) =v_1^TS^{(i)}u_1+v_2^T\hat{A}^{(i)}_{\Gamma \Gamma }u_2. 
\end{equation*}
\end{proof}
Now let us consider the local extension $\hat{R}_0^{(i)^T}: V_h(\Gamma_i)\to V_h(\Omega_i)$ defined by :
\begin{equation*}
    \hat{R}_0^{(i)^T}u_{\Gamma_i}=\begin{bmatrix} u_{\Gamma_i}  \\  \hat{P}^{(i)}\!(\hat{Q}^{(i)^T}\!\!\hat{A}^{(i)}_{\Gamma\Gamma}\hat{Q}^{(i)})^{-1}\hat{Q}^{(i)^T}\!\!\hat{A}_{\Gamma\Gamma}^{(i)}u_{\Gamma_i}
    \end{bmatrix}
\end{equation*} 
in the sense of local bilinear form \begin{equation*}
    a^{(i)}(u^{(i)},v^{(i)})=v^{(i)^T}\begin{bmatrix}
  A^{(i)}_{\Gamma\Gamma}&A^{(i)}_{\Gamma I}\\
  A^{(i)}_{I\Gamma}&A^{(i)}_{II}
\end{bmatrix}u^{(i)},
\end{equation*}
where $u^{(i)},v^{(i)}\in V_h(\Omega_i).$

\begin{lemma} \label{bound1} For $u_{\Gamma_i},v_{\Gamma_i} \in V_h(\Gamma_i)$ holds
  \[
 {a}^{(i)}(\hat{R}_0^{(i)^T} u_{\Gamma_i}, \hat{R}_0^{(i)^T} v_{\Gamma_i}) = 
 (\hat{\Pi}^{(i)}_Sv_{\Gamma_i})^T S^{(i)} (\hat{\Pi}^{(i)}_S u_{\Gamma_i})^T + ( v_{\Gamma_i} - \hat{\Pi}^{(i)}_S v_{\Gamma_i})^T
  A^{(i)}_{\Gamma \Gamma}  ( u_{\Gamma_i} - \hat{\Pi}^{(i)}_S u_{\Gamma_i}).
  \]
\end{lemma} 
\begin{proof}
Let us still consider the projection $u_1 = \hat{\Pi}^{(i)}_S u_{\Gamma_i}$ and $u_2 = u_{\Gamma_i} - u_1$ and
$v_1 = \hat{\Pi}^{(i)}_S v_{\Gamma_i}$ and $v_2 = v_{\Gamma_i} - v_1$
  . We have 
\begin{equation*}
\begin{split}
    a^{(i)}(\hat{R}_0^{(i)^T}u_{\Gamma_i},\hat{R}_0^{(i)^T}v_{\Gamma_i})&=\begin{bmatrix}v_1^T+v_2^T ,&  -v_1^TA^{(i)}_{\Gamma I}A^{(i)^{-1}}_{II} 
        \end{bmatrix}\begin{bmatrix} A^{(i)}_{\Gamma\Gamma}  &  A^{(i)}_{\Gamma I}   \\ 
     A^{(i)}_{I \Gamma } & A^{(i)}_{II}\\
        \end{bmatrix}\begin{bmatrix}u_1+u_2 \\  -A^{(i)^{-1}}_{II} A^{(i)}_{I \Gamma }u_1
        \end{bmatrix}\\
        =&v_1^TA^{(i)}_{\Gamma\Gamma}u_1+v_2^TA^{(i)}_{\Gamma\Gamma}u_1+v_1^TA^{(i)}_{\Gamma\Gamma}u_2+v_2^TA^{(i)}_{\Gamma\Gamma}u_2-v_1^TA^{(i)}_{\Gamma I}A_{II}^{(i)^{-1}}A^{(i)}_{I\Gamma}u_1\\
        &-v_2^TA^{(i)}_{\Gamma I}A_{II}^{(i)^{-1}}A^{(i)}_{I\Gamma}u_1-v_1^TA^{(i)}_{\Gamma I}A_{II}^{(i)^{-1}}A^{(i)}_{I\Gamma}u_2\\
        =&v_1^TA^{(i)}_{\Gamma\Gamma}u_1+v_2^TS^{(i)}u_1+v_1^TS^{(i)}u_2+v_2^TA^{(i)}_{\Gamma\Gamma}u_2-v_1^TA^{(i)}_{\Gamma I}A_{II}^{-1}A^{(i)}_{I\Gamma}u_1\\
        =&v_1^TA^{(i)}_{\Gamma\Gamma}u_1+v_2^TA^{(i)}_{\Gamma\Gamma}u_2-v_1^TA^{(i)}_{\Gamma I}A_{II}^{(i)^{-1}}
        A^{(i)}_{I \Gamma}u_1 = v_1^T S^{(i)} u_1 + v_2^T A^{(i)}_{\Gamma\Gamma}u_2.\\
\end{split}
\end{equation*}
\end{proof} 

Now we can prove the condition number for the inexact solver. Denoted  $\hat{\lambda}_{min}(\eta)=\displaystyle{\min_{1\leq i\leq N}}\{\hat{\lambda}_{k_i+1}^{(i)} \}$, the smallest eigenvalue which greater than $\eta$ for all subdomains. Then we check the following three key assumptions.
\begin{lemma} (Assumption i) For $u\in V_h(\Omega)$, there exist $u_i\in V_i$ for  $0\leq i\leq N$, such that $u=\hat{R}_0^Tu_0+ \sum_{i=1}^N R_{i}^{T}u_i$ and satisfies
  \[ \hat{a}_{0}(u_0,u_0) + \sum_{i=1}^N {a}_i(u_i,u_i)\leq
 (2 + 7\max\{1,\frac{1}{\hat{\lambda}_{min}(\eta)}\}) a(u,u).\]
\end{lemma}  
\begin{proof}
  The decomposition is unique given by $u_0 = u_\Gamma$ and
  the others $u_i$ obtained from
  $\sum_{i=1}^N  R_{i}^{T}u_i = u - \hat{R}_0^T  u_0$. Hence,
    \begin{equation*}
        \begin{split}
            \hat{a}_0(u_0,u_0)\!\! +\!\!\sum_{i=1}^N{a}_i(u_i,u_i)\!\!&=\!\hat{a}_0(u_0,u_0) \!\!+\!\sum_{i=1}^Na(R^{T}_{i}\!\!u_i,R^{T}_{i}\!\!u_i)\\
            &= \hat{a}_0(u_0,u_0) \!\!+\!a(u\!-\!\hat{R}_0^Tu_0,u\!-\!\hat{R}_0^Tu_0) \hspace{6pt}\text{(orthogonality of each subdomain)}\\
          & \leq \hat{a}_0(u_0,u_0)  +2a(u,u)+2a(\hat{R}_0^Tu_0,\hat{R}_0^Tu_0)\hspace{10pt} \\
          &= 3\sum_{i=1}^N\!u_1^{(i)^T} \!\!S^{(i)} u_1^{(i)}\! +
          \sum_{i=1}^N u_2^{{(i)}^T} \!\!\hat{A}^{(i)}_{\Gamma \Gamma}u^{(i)}_2+2 a(u,u)\! +
          2\sum_{i=1}^N \!u_2^{{(i)}^T}\!\! {A}^{(i)}_{\Gamma \Gamma}u^{(i)}_2.
        \end{split}
    \end{equation*}
    
   Here we use \Cref{theo3} and \Cref{bound1}, and $u_1^{(i)}=\hat{\Pi}^{(i)}_S R_{\Gamma_i\Gamma}u_0$, $u_2^{(i)}=R_{\Gamma_i\Gamma}u_0-\hat{\Pi}^{(i)}_S R_{\Gamma_i\Gamma}u_0$, for $1\leq i\leq N$.
    In case $\hat{A}^{(i)}_{\Gamma\Gamma}$ is the diagonal or block-diagonal
    of $A^{(i)}_{\Gamma \Gamma}$, by using elementwise argument and 
    Cauchy-Schwarz inequalities, we have
    $A^{(i)}_{\Gamma \Gamma} \leq 3 \hat{A}^{(i)}_{\Gamma\Gamma}$ for a general
    triangulation, or $A^{(i)}_{\Gamma \Gamma} \leq 2 \hat{A}^{(i)}_{\Gamma\Gamma}$
    for triangulation with right triangles. 
    
    Then we have:
    \begin{equation*}
    \begin{split}
          &\leq   3\sum_{i=1}^N\!u_1^{(i)^T} \!\!S^{(i)} u_1^{(i)}\! +
          7\sum_{i=1}^N u_2^{{(i)}^T} \!\!\hat{A}^{(i)}_{\Gamma \Gamma}u^{(i)}_2+2 a(u,u)\\
          & \leq 2a(u,u) + 7\max\{1,\frac{1}{\hat{\lambda}_{min}(\eta)}\} \sum_{i=1}^N (u_1^{(i)}+u_2^{(i)})^T S^{(i)} (u_1^{(i)}+u_2^{(i)})\\
          &=
         2a(u,u) + 7\max\{1,\frac{1}{\hat{\lambda}_{min}(\eta)}\} u_0^T\sum_{i=1}^N  R_{\Gamma_i\Gamma}^T 
         S^{(i)}  R_{\Gamma_i\Gamma}u_0=(2+7\max\{1,\frac{1}{\hat{\lambda}_{min}(\eta)}\})a(u,u).
    \end{split}
    \end{equation*}
\end{proof}

The following two assumptions follow similar arguments as above.
\begin{lemma} (Assumption ii) We have $\mu(\epsilon)=1$.
\end{lemma}

\begin{lemma} \label{ass2} (Assumption iii) We have 
  \[
  a(R_i^Tu_i,R_i^Tu_i)\leq  {a}_i(u_i,u_i) \hspace{20pt} \forall u_i\in V_i \quad 1\leq i\leq N,
  \]
  and \[
  a(\hat{R}_0^Tu_0,\hat{R}_0^Tu_0)\leq  3\hat{a}_0(u_0,u_0) \hspace{20pt} \forall u_0\in V_0.
  \]
\end{lemma}
\begin{proof}
  The first inequality follows from
  the definition of the $a_i(\cdot,\cdot)$ for $1\leq i\leq N$. 

 The second inequality, we define $u_1^{(i)}=\hat{\Pi}^{(i)}_S R_{\Gamma_i\Gamma}u_0$, $u_2^{(i)}=R_{\Gamma_i\Gamma}u_0-\hat{\Pi}^{(i)}_S R_{\Gamma_i\Gamma}u_0$, for $1\leq i\leq N$. Then we use \Cref{bound1}, \Cref{theo3}, and $A^{(i)}_{\Gamma \Gamma} \leq 3 \hat{A}^{(i)}_{\Gamma\Gamma}$ to get:
\begin{equation*}
    \begin{split}
        a(\hat{R}_0^Tu_0,\hat{R}_0^Tu_0) = \sum_{i=1}^Nu_1^{(i)^T}\!\! S^{(i)} u_1^{(i)} +
          \sum_{i=1}^N \!u_2^{{(i)}^T}\!\! {A}^{(i)}_{\Gamma \Gamma}u^{(i)}_2\leq 3\hat{a}_0(u_0,u_0).
    \end{split}
\end{equation*}
\end{proof}

\begin{theorem} \label{theo4} For any $u\in V_h(\Omega)$, the following holds:
  $$( 2 + 7\max\{1,\frac{1}{\hat{\lambda}_{min}(\eta)}\})^{-1}  a(u,u)\leq a(\hat{T}_Au,u) \leq 4 a(u,u),$$
where $\hat{T}_A$ defined similar as in \cref{Def_T_A} and $\hat{\lambda}_{min}(\eta)$ defined in previous page.
\end{theorem}
\begin{proof}
  It follows from the additive Schwarz theory; see \cite[Chapter 2]{MR2104179}.
\end{proof}

Also we note that similar as the exact solver, we can choose $\eta=O(h/H)$ to guarantee condition number is $O(H/h)$.

\subsection{Comparison with other methods}  \label{BDD-GenE0} \hspace*{\fill}

We compare our method with BDD-GenEO \cite{MR3093793}.
The resemblance between BDD-GenEO in \cite{MR3093793} and NOSAS is that the generalized eigenvalue problems are in $\Omega_i$. In BDD-GenEO is $D_i{S}^{(i)}D_i\xi_j^{(i)}=\lambda_j^{(i)}R_{\Gamma_i\Gamma}A_{\Gamma\Gamma}R_{\Gamma_i\Gamma}^{T}\xi_j^{(i)}$ in each subdomain $\Omega_i$, where the $D_i$ are
diagonal matrices associated to a proper partition of unity,  and
the right-hand side uses information from adjacent subdomains.
NOSAS methods neither require partition of unity nor information from adjacent
subdomains, that is, all the information needed is the
Neumann matrix $A^{(i)}$. We also note that NOSAS methods
are based on AAS, and differently from BDD-based solvers, the exact  $S^{(i)}$ are not required
when applying the preconditioned system. We note that 
  there are other versions on the literature, such as in \cite{MR3093793},
  where $\int_{\partial \Omega_i} \rho(x) u_{\Gamma_i} v_{\Gamma_i} \,ds$ is used rather than 
  $u_{\Gamma_i}^TA^{(i)}_{\Gamma\Gamma}v_{\Gamma_i}$ on the right-side of the generalized eigenvalue problem. We note this choice is different from our choice, not only by a $h$ scaling but also on the dependence of the coefficients, as showing in \Cref{NOSAS_figure3}. This difference becomes evident when proving \Cref{isolated_thm}, see below. Finally, we remark that the way we define the coarse functions on $\Gamma$ are based on $a$-minimum energy at the nodes on $\Gamma_i$ while the BDD-GenEO is based on the partition of unity. 

    To better understand how coefficients in $\Omega_i$ effects the number of eigenvalues of the 
    generalized eigenvalue problem  \cref{exact_eigen}, we next define the concept of  high-permeable island as:

\begin{definition}
  \label{island_def}

A high-permeable island ${\Omega}_{i,m} \subset  \overline{\Omega}_i$
with high-contrast coefficients is defined by: 1)  ${\Omega}_{i,m}$ is a
closed connected region (union of elements with large coefficients
$\rho_1$),  2) ${\Omega}_{i,m}$ is surrounding by elements with small
coefficients $\rho_2$.
\end{definition}

We remind that elements are closed sets and $\Gamma$ does not include any
Dirichlet node on $\partial \Omega$ and an island can be a channel or an inclusion. 
Next, we want to find the number of small eigenvalues with $O(\rho_2/\rho_1)$ when we have only two high-contrast coefficients in each subdomain. Inspired by the Appendix A of \cite{MR2728702}, we have the following theorem for two-dimensional subdomain $\Omega_i$.
\begin{theorem}
\label{isolated_thm} Assume that $\rho_1\gg\rho_2$, 
then the number of
small eigenvalues $O(\rho_2/\rho_1)$ of the generalized eigenvalue problem
  \cref{exact_eigen}
is equal to the number of high-permeable islands in \cref{island_def} that touch $\Gamma_i$ in at least one node and does not touch Dirichlet boundary. 
 \end{theorem}
\begin{proof}
 Suppose there are $\tilde{M}$ high-permeable  islands $\Omega_{i,m}$, and only $\Omega_{i,m}$ $(1\leq m\leq M)$ touch $\Gamma_i$ in at least one node and does not touch Dirichlet boundary. Consider $\lambda_1^{(i)}\leq \lambda_2^{(i)} \leq \cdots \leq \lambda^{(i)}_{n_i}$ of the generalized eigenvalue problem   \cref{exact_eigen}.

 We first present an upper bound for
 $\lambda_M^{(i)}$ using Courant--Fischer--Weyl min-max principle given by
   \begin{equation*}
    \lambda_M^{(i)}=\min_{dim(W)=M}\max_{v\,\in W\backslash \{ 0 \}}R(v) \leq
    \max_{v\,\in W^* \backslash \{ 0 \}}R(v)  \hspace{10pt}\text{where} 
    \hspace{10pt} R(v)=\frac{v^TS^{(i)}v}{v^TA_{\Gamma\Gamma}^{(i)}v}.
\end{equation*}
   Here  $W$ is any M-dimensional subspace of $V_h(\Gamma_i)$ and
  $W^*=\text{Span}\{v_1, v_2,\cdots,v_M\}$ where $v_m$ are $M$
  linearly independent vectors of $V_h(\Gamma_i)$ introduced as follows.
  Let $\Gamma_{i,m} := \partial \Omega_{i,m} \cap \Gamma_{i}$, that is, the boundary of the  high-permeable island $\Omega_{i,m}$ which touches $\Gamma_i$ and
  define  $v_m\in V_h(\Gamma_i)$ to be equal to one 
  on nodes of $\Gamma_{i,m}$ and equal to zero on the remaining nodes
  of $\Gamma_i$. We denote $\Gamma_{i,m}^\delta$ as the union of all elements of $\Omega_{i,m}$  which touch at least one node of $\Gamma_{i,m}$. Let us define $\mathcal{E}_{1,m}(v_m)\in V_h(\Omega_i)$ to be equal to
  one on the nodes of ${\Omega}_{i,m}$ for $1\leq m\leq M$ and equal to zero at the remaining
  nodes of  $V_h(\Omega_i)$. We define $\mathcal{E}_{2,m}(v_m)\in V_h(\Omega_i)$ as the zero trivial  extension of $v_m$ in $V_h(\Omega_i)$. Then given 
  $v \in W^*$ we can write as $v = \sum_{m=1}^M \alpha_m v_m$. We have
  \begin{equation*}
    v^TS^{(i)}v \leq \int_{\Omega_i}\rho(x)| \nabla
    \sum_{m=1}^M \mathcal{E}_{1,m} (\alpha_m
    v_m)|^2dx  \preceq
    \sum_{m=1}^M \alpha^2_m \rho_2|\partial \Omega_{i,m} \backslash \Gamma_{i,m}|/h,
  \end{equation*}
  where $|\partial \Omega_{i,m} \backslash \Gamma_{i,m}|$ denotes the lenght of $
  \partial \Omega_{i,m} \backslash \Gamma_{i,m}$. The first inequality follows
  from minimum energy of $a_i$-discrete harmonic extension. The second
  inequality follows from computing energy of zero extensions. The
  hidden constant of the second inequality depends only
  on the shape of the elements. We also have
 \begin{equation*}
   v^TA_{\Gamma \Gamma}^{(i)}v  = \int_{\Omega_i}\rho(x) |\nabla
    \sum_{m=1}^M \mathcal{E}_{2,m} (\alpha_m
   v_m)|^2dx\asymp    
\sum_{i=1}^M \alpha_m^2\rho_1  |\Gamma^\delta_{i,m}|/h^2.
 \end{equation*}
 We note that we have used $|\Gamma_{i,m}^\delta|/h^2$ rather than
 $|\Gamma_{i,m}|/h$ because $\Gamma_{i,m}$ might be just a node. 
 We finally obtain
 \begin{equation*}
   R(v)=\frac{v^TS^{(i)}v}{v^TA_{\Gamma\Gamma}^{(i)}v} \preceq
   \frac{\rho_2}{\rho_1}
  \frac{ \displaystyle{\max_{1\leq m\leq M}} h|\partial \Omega_{i,m} \backslash \Gamma_{i,m}|}
        {\displaystyle{\min_{1\leq m\leq M}}  |\Gamma_{i,m}^\delta|}.
 \end{equation*}
  
 Now we prove that there are at most $M$ small eigenvalues of
 $O(\rho_2/\rho_1)$. The $(M+1)$-th smallest  
 eigenvalue can be characterized via Courant--Fischer--Weyl min-max
 principle given by
\begin{equation*}
  \lambda_{M+1}^{(i)}=\max_{codim(W_c)=M}\min_{v\,\in W_c\backslash \{ 0 \}}R(v)
  \geq \min_{v\,\in W^*_c\backslash \{ 0 \}}R(v),
\end{equation*}
where $W_c$ is any subspace of $V_h(\Gamma_i)$ with codimension $M$, that is,
dimension $n_i - M$. The subspace $W^*_c$ of codimension $M$ is defined by 
\begin{equation*}
  W^*_c  =\{v\in V_h(\Gamma_i):  v(x_*^m)= 0, ~~  \text{for all}~~  1 \leq m \leq M\},
\end{equation*}
where $x_*^m$ is any selected node of $\Gamma_{i,m}$. 

Let $v \in W^*_c$, define $v_1\in V_h(\Omega_i)$ as the $a_i$-discrete harmonic extension of $v$ in $\Omega_i$ and $v_2\in V_h(\Omega_i)$ as the zero extension of $v$ in $\Omega_i$. We have 

 \begin{equation} \label{lambdaM1} 
\begin{split}
    R(v)=\frac{v^TS^{(i)}v}{v^TA_{\Gamma\Gamma}^{(i)}v}&=\frac{\int_{\Omega_i}\rho(x) |\nabla v_1|^2dx}{\int_{\Omega_i}\rho(x) |\nabla v_2|^2dx}=\frac{\int_{\Omega_i}\rho(x) |\nabla v_1|^2dx}{\int_{\Gamma_i^\delta}\rho(x) |\nabla v_2|^2dx},\\
    & = \frac{\sum_{m=1}^{\tilde{M}}(\rho_1-\rho_2)|v_1|^2_{H^1(\Omega_{i,m})}+ \rho_2|v_1|^2_{H^1(\Omega_{i})}}{\sum_{m=1}^{\tilde{M}}(\rho_1-\rho_2)|v_2|^2_{H^1(\Gamma_{i,m}^\delta)}+\rho_2|v_2|^2_{H^1(\Gamma_i^\delta)}},
    \end{split}
\end{equation}
where ${\Omega}_{i,m}$ for $1\leq m\leq \tilde{M}$ are the high-permeable islands. Among these islands we consider three types: Case 1)  ${\Omega}_{i,m}
 \cap \Gamma_i \neq \emptyset$ and ${\Omega}_{i,m}
 \cap \partial \Omega = \emptyset$; Case 2) ${\Omega}_{i,m} \cap
 \Gamma_i \neq \emptyset$ and ${\Omega}_{i,m} \cap
 \partial \Omega \neq  \emptyset$; Case 3) ${\Omega}_{i,m} \cap
 \Gamma_i = \emptyset$.
 We first consider the Case 1) for $\Omega_{i,m}$ with $1\leq m\leq M$ and later we consider the Case 2) and Case 3) for $\Omega_{i,m}$ with $M+1\leq m\leq \tilde{M}$. 

 Note that $\Gamma_{i,m} = {\Omega}_{i,m} \cap \Gamma_i$ might not
 be connected.  We first consider the case that $\Gamma_{i,m}$ is connected. Assume there are $J^m\geq1$ nodes on
 $\Gamma_{i,m}$. Then we have 
\begin{equation}  \label{localv2} 
\begin{split} 
  |v_2|^2_{H^1(\Gamma_{i,m}^\delta)}&\asymp \sum_{j=1}^{J^m}  v(x_j^m)^2 =
  \sum_{j=1}^{J^m}(v(x_j^m) - v(x^m_*))^2  \preceq (J^m)^2 \sum_{j=2}^{J^m}
(v(x_{j-1}^m) - v(x^m_j))^2 \\  
& \preceq (J^m)^2 |v_1|^2_{H^1(\Gamma_{i,m}^\delta)}   
  \preceq (|\Gamma_{i,m}|/h)^2 |v_1|^2_{H^1(\Omega_{i,m})}.
\end{split}
\end{equation} 
 
For the case $\Gamma_{i,m}$ is not connected, assume we have $K_m$ connected
components $\Gamma_{i,m}^k$ for $1 \leq k \leq K_m$. Without loss of generality,
assume $x_*^m \in  \Gamma_{i,m}^1$. Since ${\Omega}_{i,m}$ is connected,
for each $\Gamma_{i,m}^k$ let us select a node  $x_*^{m,k} \in  \Gamma_{i,m}^k$
and the shortest path from $x_*^m$ to $x_*^{m,k}$  for  $2 \leq k \leq K_m$. The
paths are graphs $G_k^m= (V_k^m,E_k^m)$ using only edges and vertices of
the triangulation on ${\Omega}_{i,m}$. Let $V_k^m$ are the
vertices of $G_k^m$ and denoted by
\begin{equation*} 
  V_k^m= \{ x_*^m = y_{1,k}^m, y_{2,k}^m ,\cdots, y_{L_k^m,k}^m =x_*^{m,k}\},
\end{equation*}
and note that $L_k^m\asymp  |G_k^m|/h$ where $|G_k^m|$ is the lenght of $G_k^m$.
Using similar arguments and using that $v(x_*^{m})= 0$ we obtain 
\begin{equation} \label{linfcontrol} 
      v(x_*^{m,k})^2 \preceq L_k^m\sum_{j=2}^{L_k^m}
  |v_1(y_{j,k}^m) - v_1(y_{j-1,k}^m)|^2 \preceq
  (|G_k^m|/h) |v_1|^2_{H^1(\Omega_{i,m})}.
\end{equation}
Using similar argument as in \cref{localv2} and using \cref{linfcontrol} we obtain 
\begin{equation}  \label{multiple} 
  \begin{split} 
    |v_2|^2_{H^1(\Gamma_{i,m}^\delta)} & \asymp \sum_{k=1}^{K_m}       \sum_{j=1}^{J^m_k}
    v(x_{j,k}^m)^2 \preceq 
    \sum_{k=1}^{K_m}  \sum_{j=1}^{J^m_k} \left(
 ( v(x_{j,k}^m) - v(x_*^{m,k}))^2 + v(x_*^{m,k})^2 \right) \\ 
    &     \preceq
\left(
(|\Gamma_{i,m}^1|/h)^2 + \sum_{k=2}^{K_m} \Big(|\Gamma_{i,m}^k|^2/h^2 + |\Gamma_{i,m}^k||G_k^m|/h^2\Big)
\right) |v_1|^2_{H^1(\Omega_{i,m})}.
\end{split} 
\end{equation}

We also should consider the Case 2) $\Omega_{i,m}$ for some $M+1 \leq
m \leq \tilde{M}$. Let $\Gamma_{i,m}^k$ for $1 \leq k \leq K_m$ be the connected components of
${\Omega}_{i,m} \cap \Gamma_i$. Now consider any node $x_*^m
\in {\Omega}_{i,m} \cap \partial \Omega$. Note that
the constraint $v(x_*^m) = 0$ is automatically satisfied without
imposing extra constraints in $W_c^*$. The same arguments above hold
by creating paths
from $x_*^{m,k}$ to  $x_*^{m}$ for $1\leq k\leq K_m$.  The Case 3) is easy to treat since $\Gamma_{i,m}^\delta=\emptyset$.

   Now we consider the last term  of \cref{lambdaM1} given by  
\begin{equation} \label{esti2} 
  |v_2|^2_{H^1(\Gamma_{i}^\delta)} \asymp \sum_{j=1}^{n_i}  v(x_j)^2 \preceq 
  n_i \max_{1\leq j\leq n_i} v(x_j)^2 \preceq (|\Gamma_i|/h)(1 + \log(\Gamma_i/h))
  |v_1|^2_{H^1(\Omega_{i})}.
\end{equation}
The last inequality holds because $v$ vanishes at the constraints nodes
$x^m_*$ and then we can use classical finite element $L_\infty$ bounds
in terms of energy for well-shaped polygonal domains of size $O(H)$.

Above all, from \cref{lambdaM1} and using the estimates \cref{localv2}, \cref{esti2} and \cref{multiple}
we obtain
\begin{equation} \label{boundci} 
    \frac{1}{\lambda_{M+1}^{(i)}} \preceq
    \max \Big\{\frac{H}{h}(1+\log\frac{H}{h}), \max_{1\leq m \leq\tilde{M}} \sum_{k=1}^{K_m} \Big(|\Gamma_{i,m}^k|/h)^2 +  |\Gamma_{i,m}^k||G_k^m|/h^2 \Big) \Big\},
\end{equation}
where we using the notation $|G_1^m| = 0$ for $1 \leq m \leq \tilde{M}$. Note that $\lambda_{M+1}^{(i)}$ does not depend on $\rho_1$ and $\rho_2$, it depends on the geometry of the high-permeable islands and the
size and shape of the elements and the subdomain $\Omega_i$.
\end{proof}
\begin{remark} \label{cHh} 
  We can obtain better bound in \cref{esti2} if the add the constraint that
  the average of $v$ on $\partial \Omega_i$ is zero. Indeed, by using a
  Poincar\'e inequality on $\partial \Omega_i$ we have 
\begin{equation*} 
  |v_2|^2_{H^1(\Gamma_{i}^\delta)} \asymp h^{-1} \|v\|_{L^2(\partial \Omega_{i})}^2
  \preceq (H/h) |v|_{H^{1/2}(\partial \Omega_i)}^2
   \preceq (H/h) |v|_{H^1(\Omega_{i})}^2
\end{equation*}
\end{remark}

We note that \cref{isolated_thm} holds for three-dimensional problems in
the sense that we obtain a lower bound for $\lambda_{M+1}^{(i)}$ which does not
depend on $\rho_1$ and $\rho_2$, however, the dependence on the geometries of
$\Omega_{i,m}$ and $h$ in \cref{boundci} changes. We also note that the
\cref{isolated_thm} holds also for the inexact version \cref{eigeninexact}
since $A_{\Gamma \Gamma}^{(i)}$ is spectrally equivalent to
$\hat{A}_{\Gamma \Gamma}^{(i)}$.


\section{NUMERICAL EXPERIMENTS}\label{numerical}\hspace*{\fill}

We now present results of problem \cref{elliptic} for a square domain $\Omega=(0,1)^2$ with $f=1$. We divide the square domain $\Omega$ into $(1/H)^2$
congruent square subdomains with $1/H$ is an integer. Then we divide each subdomain
into $(H/h)^2$ congruent squares where $H/h$ is an integer and divides
each of these squares into two right triangle elements. We imposed zero
Dirichlet boundary condition on $\partial \Omega$. We use the Preconditioned Conjugate Gradient method (PCG) and the number of iterations required to reduce the residual by $10^{-6}$.

We first study a proper threshold $\eta$, if  $\eta$ is chosen too small,
the condition number of the preconditioned system will be large, and if $\eta$ is chosen too large, the number of the
eigenvectors might be large and hence will increase the cost of the
coarse problem.  For constant coefficients, we know from \Cref{3.3.1} that
the smallest non-zero generalized eigenvalue is $O(\frac{h}{H})$. Numerically we also see the asymptotic behavior
to $c\frac{h}{H}$. The goal is to find a good constant "c" 
to consider $\eta = c\frac{h}{H}$ as the threshold for the
heterogeneous coefficients cases. \Cref{NOSAS_figure1} shows the generalized eigenvalue of constant coefficients in each subdomain, for the exact
NOSAS with \cref{exact_eigen}, $c=0.5$, $c=1.3$ and $c=3.2$ are
three interesting choices. The choice $c=0.5$ selects only the one
eigenvalue on the floating subdomains and none for the edge and corner
subdomains, $c=1.3$ selects one eigenvalue
for each subdomain, while $c=3.2$ selects four eigenvalues for
floating subdomains, two eigenvalues for edge subdomains and one
eigenvalue for corner subdomains; see also that the eigenvalues associated to
$c=3.2$  are a little
bit isolated from the rest. We note that the zero eigenvalues
of the floating subdomains are not plotted in \Cref{NOSAS_figure1}. For the
inexact NOSAS \cref{eigeninexact}, we obtain $c=0.25$, $c=0.64$ and $c=1.6$,
respectively.

\begin{figure}[tbhp]
\centering \subfloat[$H/h=8$]{\includegraphics[width=.5\linewidth]{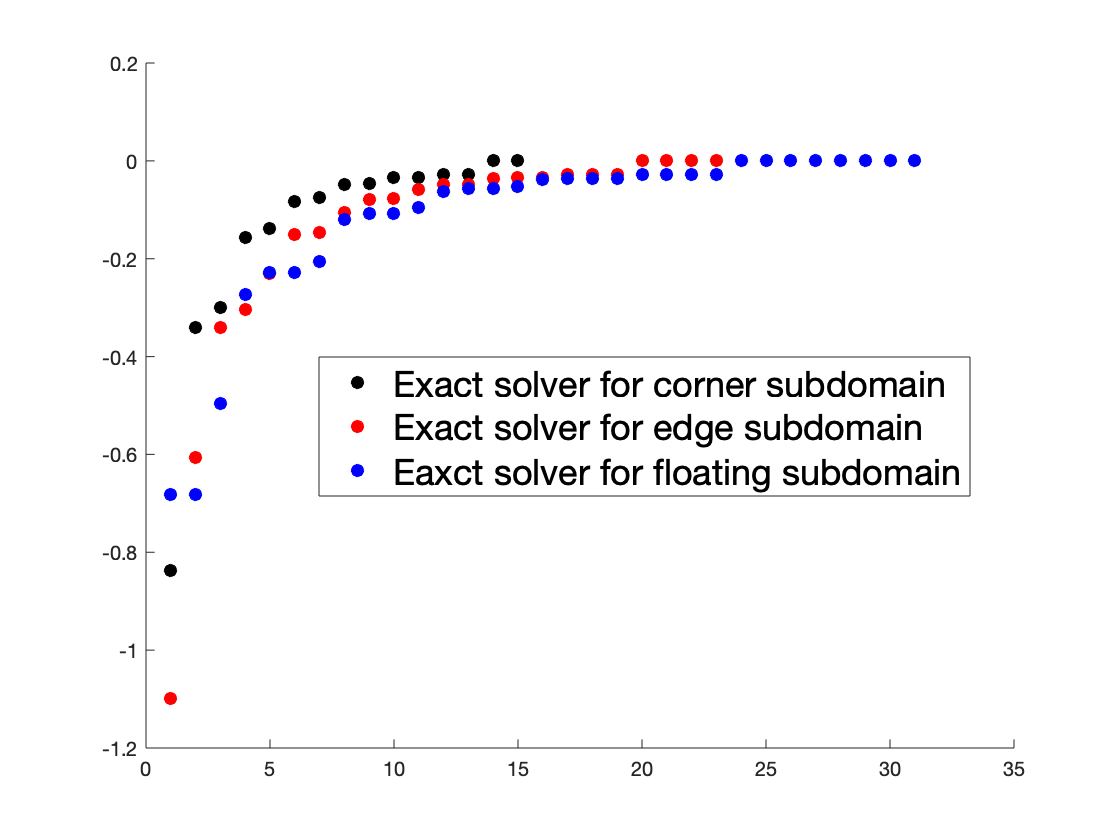}\hspace{-20pt} \includegraphics[width=.5\linewidth]{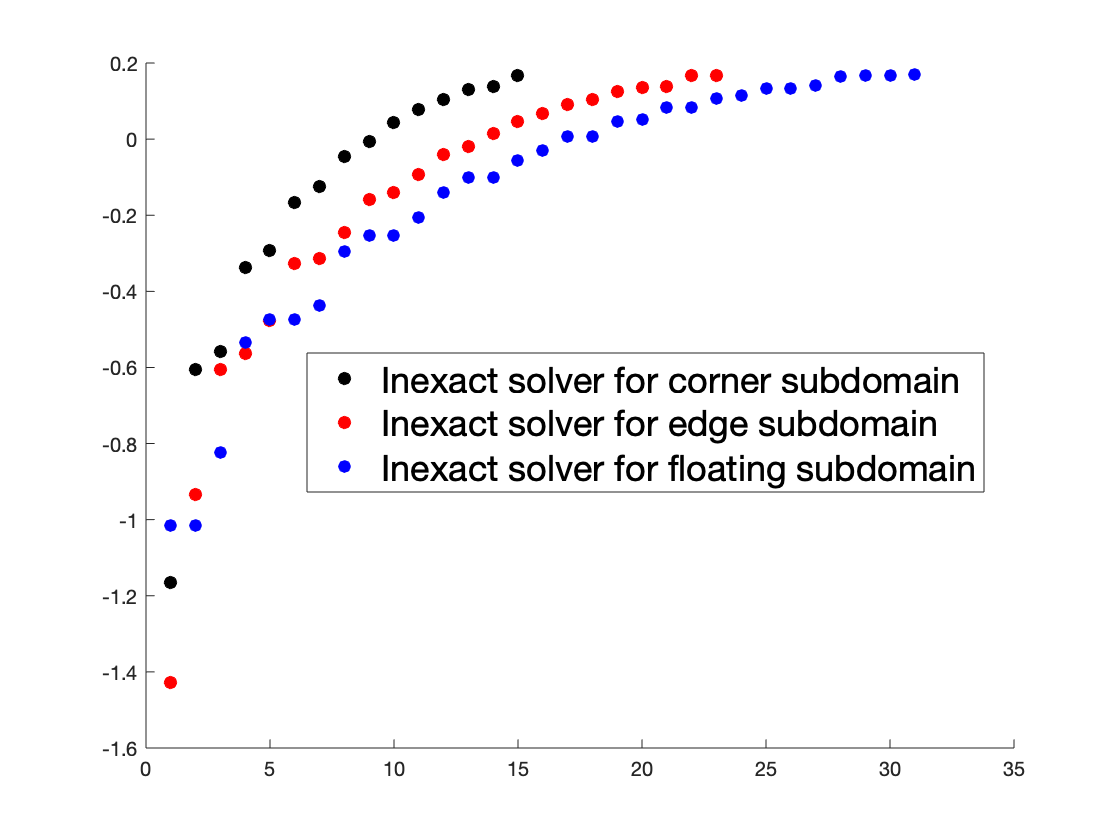}}\\

\subfloat[$H/h=16$]{\includegraphics[width=.5\linewidth]{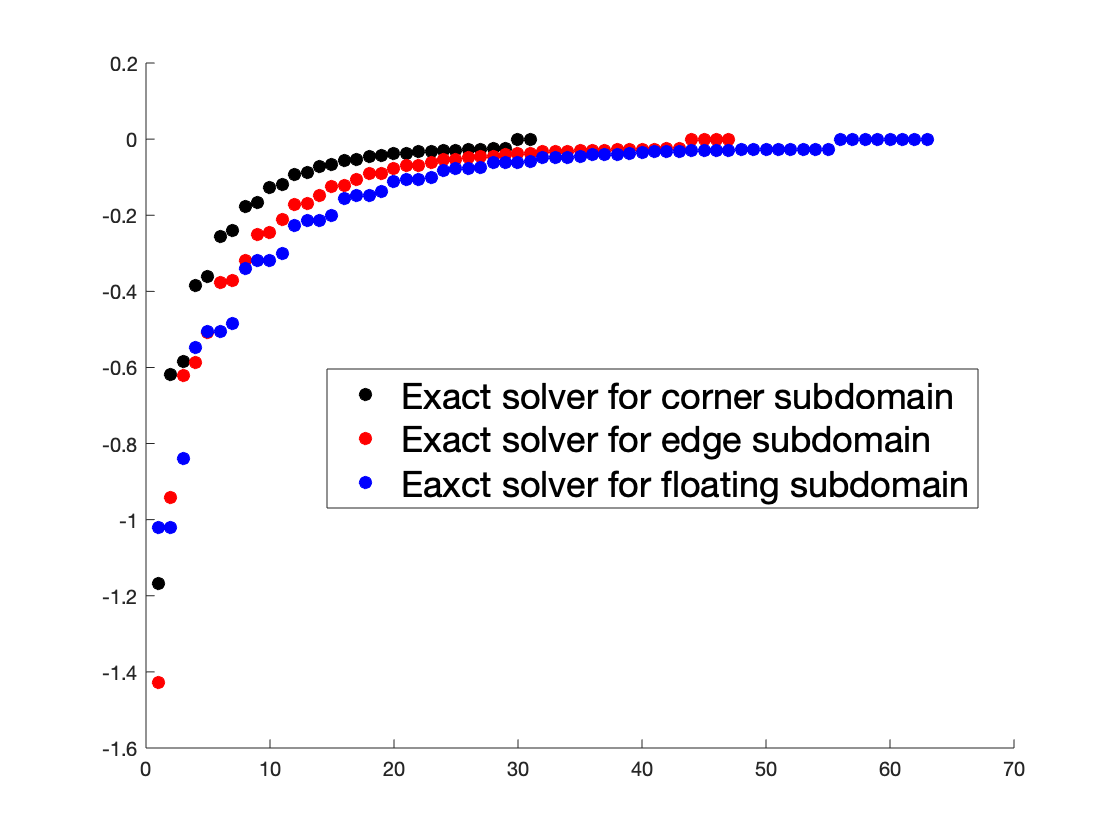}\hspace{-20pt} \includegraphics[width=.5\linewidth]{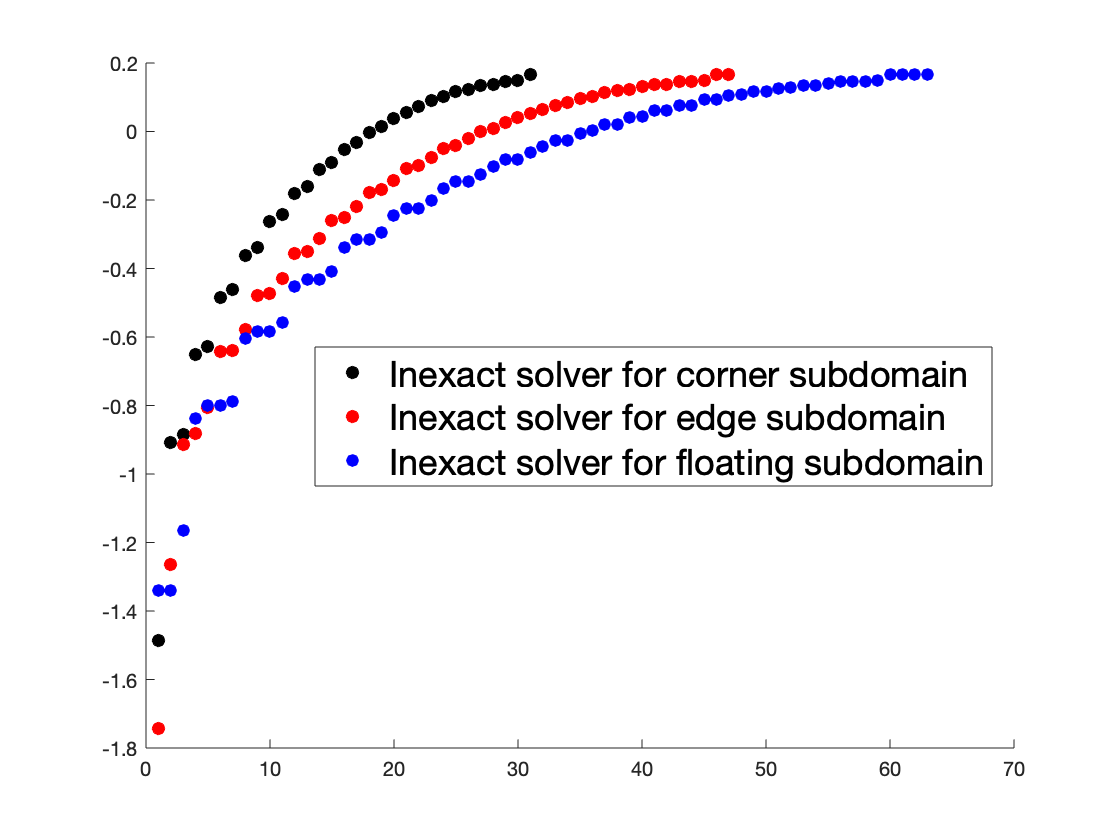}}\\
\caption{The $\log_{10}$ values of the generalized eigenvalues for
  different types of subdomain with constant coefficients. On the
  left/right are the exact/inexact versions with
  $A_{\Gamma \Gamma}^{(i)}$/$\hat{A}_{\Gamma \Gamma}^{(i)}$. On the
  top/bottom are the cases where $\frac{H}{h}=8$/$\frac{H}{h}=16$.}
\label{NOSAS_figure1}
\end{figure}

The next set of experiments is to confirm \Cref{isolated_thm}. \cref{NOSAS_table1} shows  the
smallest four generalized eigenvalues of  \cref{exact_eigen} for
two different types of subdomains: comb-like structure in \cref{NOSAS_figure2}
and string-like structure in \cref{NOSAS_figure3}. As expected by \cref{isolated_thm}, for the subdomain $\#1$ in the comb-like structure has two
high-permeable islands touching $\Gamma_i$ and not touching the Dirichlet boundary $\partial \Omega$, therefore, it results in two small eigenvalues,  and
for the subdomain $\#2$ just one small eigenvalue
since one of the high-permeable island touching
$\partial \Omega$. For the subdomain $\#1$ in the string-like structure, there is only one small eigenvalue, while for the subdomain $\#2$, there are three small
eigenvalues. See that the small eigenvalues are zero or are proportional to $\rho_2/\rho_1$. 

\begin{figure}[tbhp]
\centering
\includegraphics[width=0.5\textwidth]{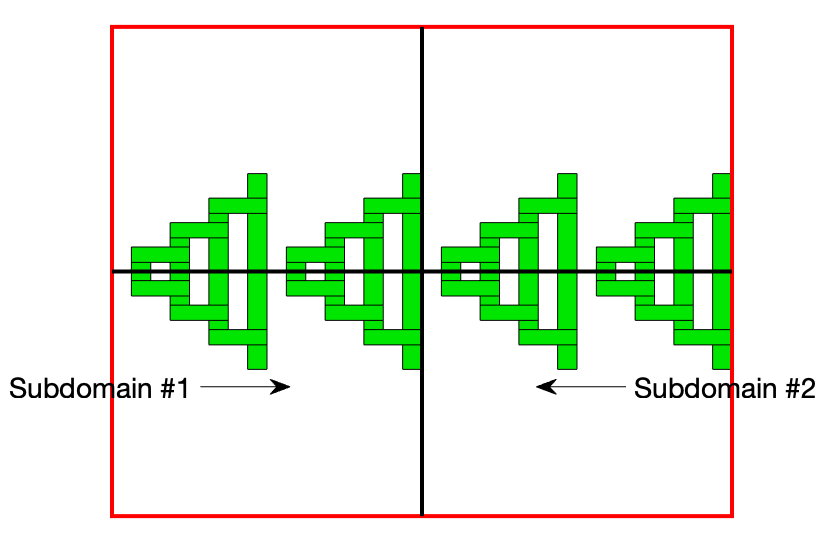}
\caption{Coefficients $\rho(x)=10^{6}$ in green areas and $\rho(x)=1$ in white areas with $H=1/2$.}
\label{NOSAS_figure2}
\end{figure}

\begin{figure}[tbhp]
\centering
\subfloat[Subdomain $\# 1$]{\label{fig:7.1}\includegraphics[width=.3\linewidth]{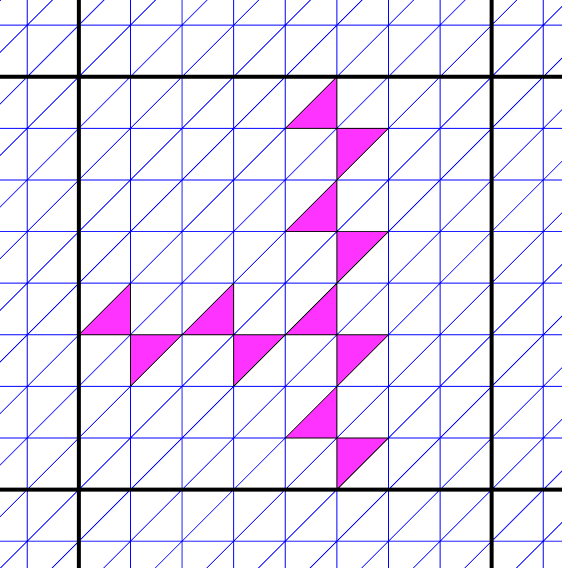}} \hspace{10pt}\subfloat[Subdomain $\# 2$]{\label{fig:7.2}\includegraphics[width=.3\linewidth]{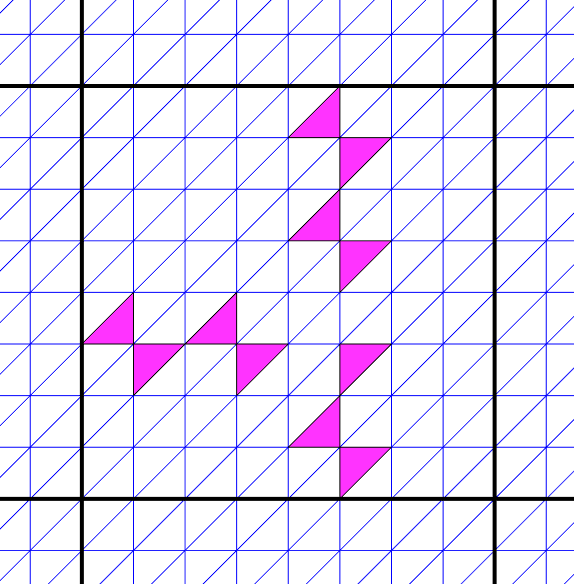}}
\caption{Coefficients $\rho(x)=10^{12}$ in magenta elements and $\rho(x)=1$ in white elements.}
\label{NOSAS_figure3}
\end{figure}

\begin{table}[tbhp]  
    \centering
\begin{tabular}{ |c|c|c|c|c|  }

 \hline
  & $\lambda_1^{(i)}$ &$\lambda_2^{(i)}$ & $\lambda_3^{(i)}$&$\lambda_4^{(i)}$\\
 \hline
\small Subdomain $\#1$ in \cref{NOSAS_figure2}&\small$2.14\times10^{-7}$&\small$1.68\times10^{-6}$&0.0907&$0.1572$\\
\hline
\small Subdomain $\#2$ in \cref{NOSAS_figure2}&\small$8.71\times10^{-7}$&$  0.0594$&0.1572&0.2296\\
\hline
\small Subdomain $\#1$ in \cref{NOSAS_figure3} & 0 &0.2000 & 0.2727 &0.3072\\
\hline
\small Subdomain $\#2$ in \cref{NOSAS_figure3} & 0 &\small$1.24\times10^{-11}$ & \small$1.51\times10^{-11}$ &0.3072\\
\hline
\end{tabular}
\caption{List of the generalized eigenvalues for the exact solver with $H/h=16$ in \cref{NOSAS_figure2} and $H/h=8$ in \cref{NOSAS_figure3}.}
\label{NOSAS_table1}
\end{table}

We now consider the example given by \Cref{NOSAS_figure4}. The width of the
green channel with coefficients $\rho=10^{6}$ across $\Omega$ is $h$,
and $\rho=1$ in the remaining areas. The left position of the channel is $\frac{1}{4}H$ away from $\Gamma_i$.
\cref{NOSAS_table2} shows the smallest three generalized eigenvalues of
a floating domain that contains the channel. The second smallest eigenvalue
behaves like $\asymp \frac{h}{H(1+\log \frac{H}{h})}$ while the third smallest eigenvalue
behaves like $\asymp \frac{h}{H}$. This agrees with the bounds \cref{boundci}
in \cref{isolated_thm} and \cref{cHh}. \cref{NOSAS_table3} shows that MES
can handle the situation that in each subdomain  there is at most one
high-permeable island that touches $\Gamma_i$ and with condition number $O(H/h)$. MES method only
deteriorates when there are more than one high-permeable islands that touch $\Gamma_i$ such as in \Cref{NOSAS_figure5}.  In that case, MES method can not get a small condition number. Instead, we use NOSAS methods.

\begin{figure}[tbhp] 
 \centering
\includegraphics[width=0.4\textwidth]{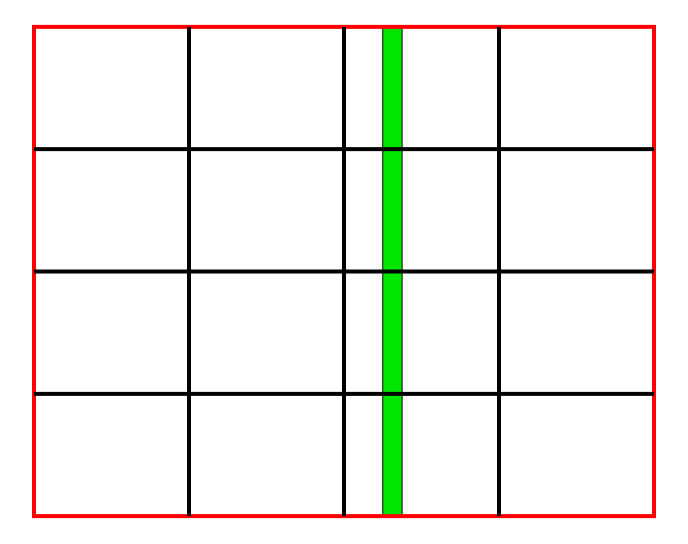} 
\caption{Coefficients $\rho=10^{6}$ in green channel and  $\rho=1$ in other areas with $H=1/4$.} \label{NOSAS_figure4} 
\end{figure}

\begin{table}[tbhp]  
    \centering
\begin{tabular}{ |c|c|c|c|c|c|  }
 \hline
  & $\lambda^{(i)}_1$ &$\lambda^{(i)}_2$ & $\lambda^{(i)}_3$&$\cdots$&$\lambda^{(i)}_{n_i}$\\
 \hline
 \multirow{2}{3.2em}{$\frac{H}{h}=8$} & 0 & 0.1548 & 0.2500&$\cdots$&1\\
& (0) & (0.0719)  &(0.1250)&$\cdots$&(1.4724)\\
 \hline
  \multirow{2}{3.2em}{\small$\frac{H}{h}=16$} & 0 & 0.0630 & 0.1250&$\cdots$&1\\
& (0) & (0.0302) &(0.0595) &$\cdots$&(1.4707)\\
\hline
 \multirow{2}{3.2em}{\small $\frac{H}{h}=32$} & 0 & 0.0284 & 0.0583&$\cdots$&1\\
& (0) & (0.0139) &(0.0282) &$\cdots$&(1.4706)\\
\hline
\end{tabular}
\\
\hspace*{5pt}\\
\caption{The eigenvalues of the floating subdomain with a green channel in the mesh of \cref{NOSAS_figure4} for exact solver and inexact solver in
parenthesis.}
\label{NOSAS_table2}
\end{table}

\begin{table}[tbhp]  
    \centering
\begin{tabular}{ |c|c|c|c|c|  }

 \hline
 MES & $H=\frac{1}{2}$ &$H=\frac{1}{4}$ & $H=\frac{1}{8}$ &$H=\frac{1}{16}$\\
 \hline
 \multirow{2}{3.2em}{$\frac{H}{h}=4$} & 13 & 23 & 28& 28\\
& (7.00) & (9.39) &(10.88) &(11.69)\\
 \hline
  \multirow{2}{3.2em}{$\frac{H}{h}=8$} & 21 & 31 & 34& 36\\
& (13.92) & (16.31) &(16.93) &(17.09)\\
\hline
 \multirow{2}{3.2em}{\small $\frac{H}{h}=16$} & 31 & 56 & 60 & 60\\
& (29.96) & (51.38
) &(60.25) &(64.74)\\
\hline
\end{tabular}
\\
\hspace*{5pt}\\
\caption{MES in the mesh of \cref{NOSAS_figure4}, the number of iterations of the PCG and the condition number in
parenthesis.}
\label{NOSAS_table3}
\end{table}

\subsection{The spectral cases}\hspace*{\fill}

In \Cref{NOSAS_figure5}, each subdomain contains two horizontal 
and two vertical white channels with low permeability $\rho(x)=1$ and
the remaining of the domain are green inclusions with $\rho(x)=10^{6}$.
For the corner subdomains, we have three high-permeable islands that touch $\Gamma_i$ and not $\partial \Omega$, corresponding to three small eigenvalues of $O(10^{-6})$. For the edge subdomains, we have five high-permeable islands that touch $\Gamma_i$ and not $\partial \Omega$, corresponding to five small eigenvalues with $O(10^{-6})$. For the floating subdomains, we have eight high-permeable islands that touch $\Gamma_i$, corresponding to eight small eigenvalues with $O(10^{-6})$. \Cref{NOSAS_table4} shows numerical results for the mesh of \cref{NOSAS_figure5}. By choosing $\eta=0.25\frac{h}{H}$ for all the NOSAS methods, the coarse problem will include all the eigenvectors associated to 
small eigenvalues. Therefore, we can expect the condition number is $O(\frac{H}{h}(1 +\log\frac{H}{h}))$.  We can see in \Cref{NOSAS_table4} that the use of inexact solvers does not deteriorate much the performance of the methods. We also note that if we put coefficients $\rho(x)=10^6$ in the white channels and $\rho(x)=1$ in the green inclusions, there would be only one high-permeable island in each subdomain. In this situation, the MES would also be satisfying. 

\begin{figure}[tbhp] 
  \centering
\includegraphics[width=0.4\textwidth]{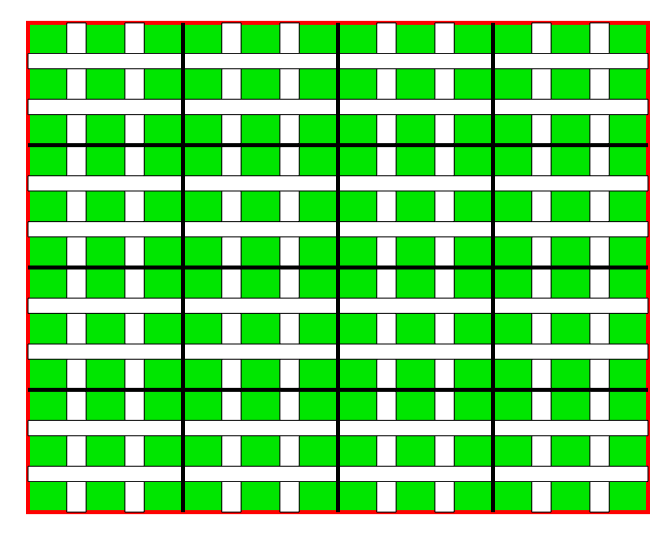}

\caption{Coefficients $\rho(x)=10^6$ in green inclusions and $\rho(x)=1$ in white channels with $H=1/4$. We fix the structure in each subdomain, which means for a different H, we still have four channels in each subdomain and nine inclusions.}
\label{NOSAS_figure5}
\end{figure}

\begin{table}[tbhp]
    \centering

    \scalebox{0.9}{
\begin{tabular}{|l|l|l|l|l|}
\hline
\begin{tabular}[c]{@{}l@{}}  \hspace{20pt}NOSAS\\             with exact solver\end{tabular} & $H=\frac{1}{2}$                                      & $H=\frac{1}{4}$                                      & $H=\frac{1}{8}$                                      & $H=\frac{1}{16}$                                     \\ \hline
\hspace{20pt}\footnotesize$\frac{H}{h}=8$                                                                                            & \begin{tabular}[c]{@{}l@{}}\hspace{6pt}9\\ (4.76)\end{tabular}   & \begin{tabular}[c]{@{}l@{}}\hspace{5pt}10\\ (4.76)\end{tabular}  & \begin{tabular}[c]{@{}l@{}}\hspace{5pt}11\\ (4.76)\end{tabular}  & \begin{tabular}[c]{@{}l@{}}\hspace{5pt}11\\ (4.76)\end{tabular}  \\ \hline
\hspace{20pt}\footnotesize$\frac{H}{h}=16$                                                                                           & \begin{tabular}[c]{@{}l@{}}\hspace{5pt}13\\ (9.74)\end{tabular}  & \begin{tabular}[c]{@{}l@{}}\hspace{5pt}16\\ (9.74)\end{tabular}  & \begin{tabular}[c]{@{}l@{}}\hspace{5pt}16\\ (9.74)\end{tabular}  & \begin{tabular}[c]{@{}l@{}}\hspace{5pt}16\\ (9.74)\end{tabular}  \\ \hline
\hspace{20pt}\footnotesize$\frac{H}{h}=32$                                                                                           & \begin{tabular}[c]{@{}l@{}}\hspace{6pt}19\\ (20.53)\end{tabular} & \begin{tabular}[c]{@{}l@{}}\hspace{6pt}25\\ (20.53)\end{tabular} & \begin{tabular}[c]{@{}l@{}}\hspace{6pt}25\\ (20.53)\end{tabular} & \begin{tabular}[c]{@{}l@{}}\hspace{6pt}25\\ (20.53)\end{tabular} \\ \hline
\end{tabular}}
\vspace*{3pt}\\
\scalebox{0.85}{
\begin{tabular}{|l|l|l|l|l|}
\hline
\begin{tabular}[c]{@{}l@{}}\hspace{40pt}NOSAS\\ \hspace{-3pt}with block diagonal inexact solver\end{tabular} & $H=\frac{1}{2}$                                      & $H=\frac{1}{4}$                                      & $H=\frac{1}{8}$                                      & $H=\frac{1}{16}$                                     \\ \hline
\hspace{40pt}\footnotesize$\frac{H}{h}=8$                                                                                         & \begin{tabular}[c]{@{}l@{}}\hspace{5pt}10\\ (4.76)\end{tabular}   & \begin{tabular}[c]{@{}l@{}}\hspace{5pt}12\\ (4.76)\end{tabular}  & \begin{tabular}[c]{@{}l@{}}\hspace{5pt}12\\ (4.76)\end{tabular}  & \begin{tabular}[c]{@{}l@{}}\hspace{5pt}12\\ (4.76)\end{tabular}  \\ \hline
\hspace{40pt}\footnotesize$\frac{H}{h}=16$                                                                                        & \begin{tabular}[c]{@{}l@{}}\hspace{5pt}14\\ (9.74)\end{tabular} & \begin{tabular}[c]{@{}l@{}}\hspace{5pt}17\\ (9.74)\end{tabular} & \begin{tabular}[c]{@{}l@{}}\hspace{5pt}17\\ (9.74)\end{tabular} & \begin{tabular}[c]{@{}l@{}}\hspace{5pt}17\\ (9.74)\end{tabular} \\ \hline
\hspace{40pt}\footnotesize$\frac{H}{h}=32$                                                                                        & \begin{tabular}[c]{@{}l@{}}\hspace{6pt}21\\ (20.53)\end{tabular} & \begin{tabular}[c]{@{}l@{}}\hspace{6pt}26\\ (20.53)\end{tabular} & \begin{tabular}[c]{@{}l@{}}\hspace{6pt}26\\ (20.53)\end{tabular} & \begin{tabular}[c]{@{}l@{}}\hspace{6pt}25\\ (20.53)\end{tabular} \\ \hline
\end{tabular}}
\vspace*{3pt}\\
\scalebox{0.9}{
\begin{tabular}{|l|l|l|l|l|}
\hline
\begin{tabular}[c]{@{}l@{}}\hspace{40pt}NOSAS\\ \hspace{-3pt}with diagonal inexact solver\end{tabular} & $H=\frac{1}{2}$                                      & $H=\frac{1}{4}$                                      & $H=\frac{1}{8}$                                      & $H=\frac{1}{16}$                                     \\ \hline
\hspace{40pt}\footnotesize$\frac{H}{h}=8$                                                                                         & \begin{tabular}[c]{@{}l@{}}\hspace{6pt}9\\ (6.47)\end{tabular}   & \begin{tabular}[c]{@{}l@{}}\hspace{5pt}11\\ (6.47)\end{tabular}  & \begin{tabular}[c]{@{}l@{}}\hspace{5pt}12\\ (6.47)\end{tabular}  & \begin{tabular}[c]{@{}l@{}}\hspace{5pt}12\\ (6.47)\end{tabular}  \\ \hline
\hspace{40pt}\footnotesize$\frac{H}{h}=16$                                                                                        & \begin{tabular}[c]{@{}l@{}}\hspace{5pt}15\\ (13.46)\end{tabular} & \begin{tabular}[c]{@{}l@{}}\hspace{5pt}18\\ (13.46)\end{tabular} & \begin{tabular}[c]{@{}l@{}}\hspace{5pt}18\\ (13.46)\end{tabular} & \begin{tabular}[c]{@{}l@{}}\hspace{5pt}19\\ (13.46)\end{tabular} \\ \hline
\hspace{40pt}\footnotesize$\frac{H}{h}=32$                                                                                        & \begin{tabular}[c]{@{}l@{}}\hspace{6pt}22\\ (28.06)\end{tabular} & \begin{tabular}[c]{@{}l@{}}\hspace{6pt}27\\ (28.06)\end{tabular} & \begin{tabular}[c]{@{}l@{}}\hspace{6pt}27\\ (28.06)\end{tabular} & \begin{tabular}[c]{@{}l@{}}\hspace{6pt}28\\ (28.06)\end{tabular} \\ \hline
\end{tabular}}
\\
\hspace*{3pt}\\
\caption{Comparison of NOSAS with exact solvers and inexact solvers in the mesh of \cref{NOSAS_figure5}. The number of iterations of the PCG and the condition number in parenthesis.}
\label{NOSAS_table4}
\end{table}

Next, we consider the size of global component of the coarse problem for the NOSAS method with the diagonal solvers. We divide the square domain into $1/H^2$ congruent square subdomain. We
select two floating subdomains $\Omega_i$, $\Omega_j$ that touch the node at $(1/2,1/2)$ like
in \cref{NOSAS_figure6}. One subdomain has two horizontal channels and two vertical channels without touching the subdomain boundary. The other subdomain has two horizontal channels and two vertical channels that touch the subdomain boundary. The coefficients $\rho(x)=10^{6}$ in green areas and $\rho(x)=1$ in each channel. According to \cref{isolated_thm}, there are no small eigenvalues in the corner and edges subdomain, and only the zero eigenvalue is small for 
the floating subdomains without any channels. For the floating
subdomain $\Omega_i$ where the channels do not touch $\Gamma_i$, there is
only one small eigenvalue, and for the floating subdomain $\Omega_j$ with channels that touch $\Gamma_j$, there are
eight small eigenvalues of size $O(10^{-6})$. Note that the size of the global component is the total number of
eigenfunction, which equal to the number of floating subdomain $+7$. 
\cref{NOSAS_table5} shows the correct total number of eigenvalues for NOSAS with the diagonal solver and threshold $\eta=0.25 \frac{h}{H}$.

\begin{figure}[tbhp]
\centering
\includegraphics[width=0.4\textwidth]{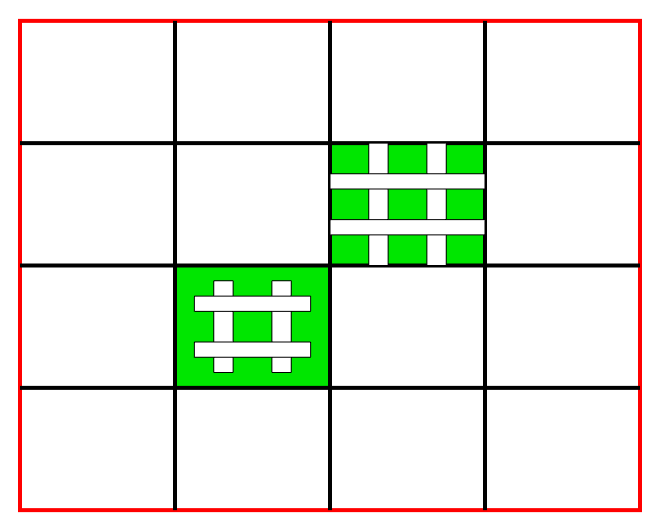}
\caption{Coefficients $\rho(x)=10^{6}$ in green areas and $\rho(x)=1$ in channels and white areas with $H=1/4$.}
\label{NOSAS_figure6}
\end{figure}
\begin{table}[tbhp]
    \centering
\scalebox{0.9}{
\begin{tabular}{ |c|c|c|c| }
 \hline
  NOSAS & iterations & condition number & number of small eigenvalues \\
  with the diagonal solver & & & \\
   \hline
  ${\scriptstyle H=\frac{1}{4},  \frac{H}{h}=8 }$&52 & 71.93  & 11\\
   \hline
  ${\scriptstyle H=\frac{1}{8},  \frac{H}{h}=8}$& 72 & 68.25  & 43\\
   \hline
 ${\scriptstyle H=\frac{1}{16},  \frac{H}{h}=8}$& 72 & 66.71  & 203\\
\hline
\end{tabular}}
\\
\hspace*{5pt}\\
 \caption{NOSAS methods and the size of the global problem for the mesh of  \cref{NOSAS_figure6}.}
 \label{NOSAS_table5}
\end{table}

We now consider NOSAS methods for \cref{NOSAS_figure7} by successively adding
very high-permeable channels to \cref{NOSAS_figure5}. \cref{NOSAS_table6} shows the good performance of NOSAS and the small dimension of the global problem with different choices of threshold $\eta$. Finally, we show the generality of NOSAS methods for the SPE10 meshes in \cref{NOSAS_figure8}, and the corresponding
good numerical results in \cref{NOSAS_table7}.

{\bf Acknowledgements.} The authors would like to thank the reviewers for their thoughtful comments and efforts towards improving our manuscript.

\begin{figure}[tbhp]
\centering \subfloat[One channel]{\label{fig:6.1}\includegraphics[width=.4\linewidth]{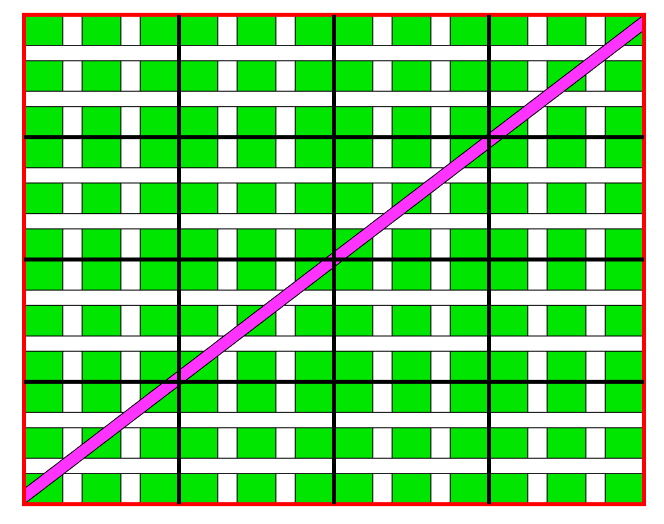}} \subfloat[Two channels]{\label{fig:6.2}\includegraphics[width=.4\linewidth]{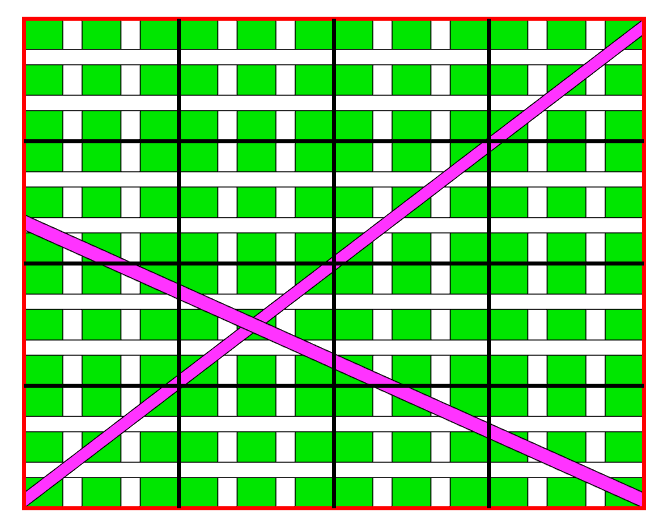}}\\

\subfloat[Three channels]{\label{fig:6.3}\includegraphics[width=.4\linewidth]{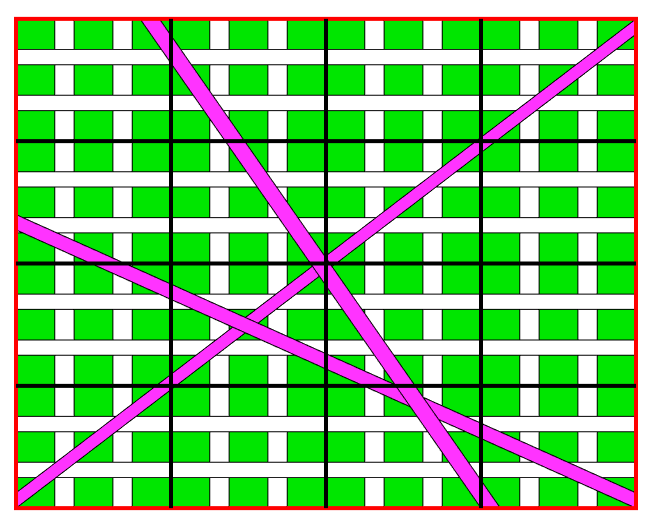}} \subfloat[Four channels]{\label{fig:6.4}\includegraphics[width=.4\linewidth]{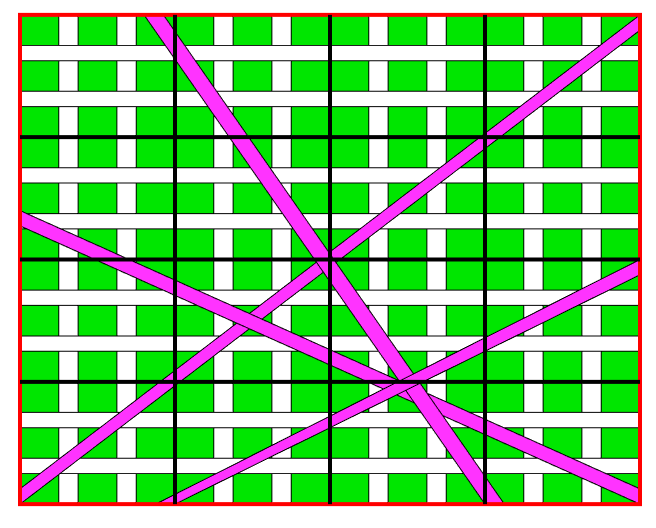}}

\caption{Adding channels to the mesh of \cref{NOSAS_figure5} with $H=1/4$, coefficients $\rho(x)=1$ in white channels, $\rho(x)=10^6$ in green inclusions, and $\rho(x)=10^{12}$ in magenta channels.}
\label{NOSAS_figure7}
\end{figure}

\begin{table}[tbhp]
    \centering
   
\begin{tabular}{ |c|c|c|c| }
 \hline
  $\eta=0.25\frac{h}{H}$ & iterations & condition number & number of small eigenvalues \\
   \hline
   No channel
  &18 & 13.46  & 84\\
   \hline
1 channel & 36 & 177.58 & 80\\
   \hline
2 channels & 87 & 155.53 & 70\\
\hline
3 channels & 95 & 147.28 & 64\\
\hline
4 channels & 102 & 148.97 & 58\\
\hline
\end{tabular}
\vspace*{3pt}\\
\begin{tabular}{ |c|c|c|c| }
 \hline
   $\eta=0.64\frac{h}{H}$ & iterations & condition number & number of small eigenvalues \\
   \hline
   No channel
  &18 & 13.46  & 84\\
   \hline
1 channel & 24 & 14.42 & 84\\
   \hline
2 channels & 57 & 53.28  & 77\\
\hline
3 channels & 63 & 59.48  & 73\\
\hline
4 channels & 67 & 59.47  & 69\\
\hline
\end{tabular}
\vspace*{3pt}\\
\begin{tabular}{ |c|c|c|c| }
 \hline
  $\eta=1.60\frac{h}{H}$  & iterations & condition number & number of small eigenvalues \\
   \hline
   No channel
  &18 & 13.46  & 84\\
   \hline
1 channel & 24 & 14.42  & 84\\
   \hline
2 channels & 41 & 25.32  & 82\\
\hline
3 channels & 39 & 25.08  & 84\\
\hline
4 channels & 41 & 25.08 & 84\\
\hline
\end{tabular}
 \caption{Choosing different threshold $\eta$ for NOSAS with inexact diagonal solver for the mesh of  \cref{NOSAS_figure7} with $H=1/4$ and $h=1/64$. }
 \label{NOSAS_table6}
\end{table}

\begin{figure}[tbhp]
\centering 
\subfloat[$Kxx\_06$]{\includegraphics[width=0.75\linewidth]{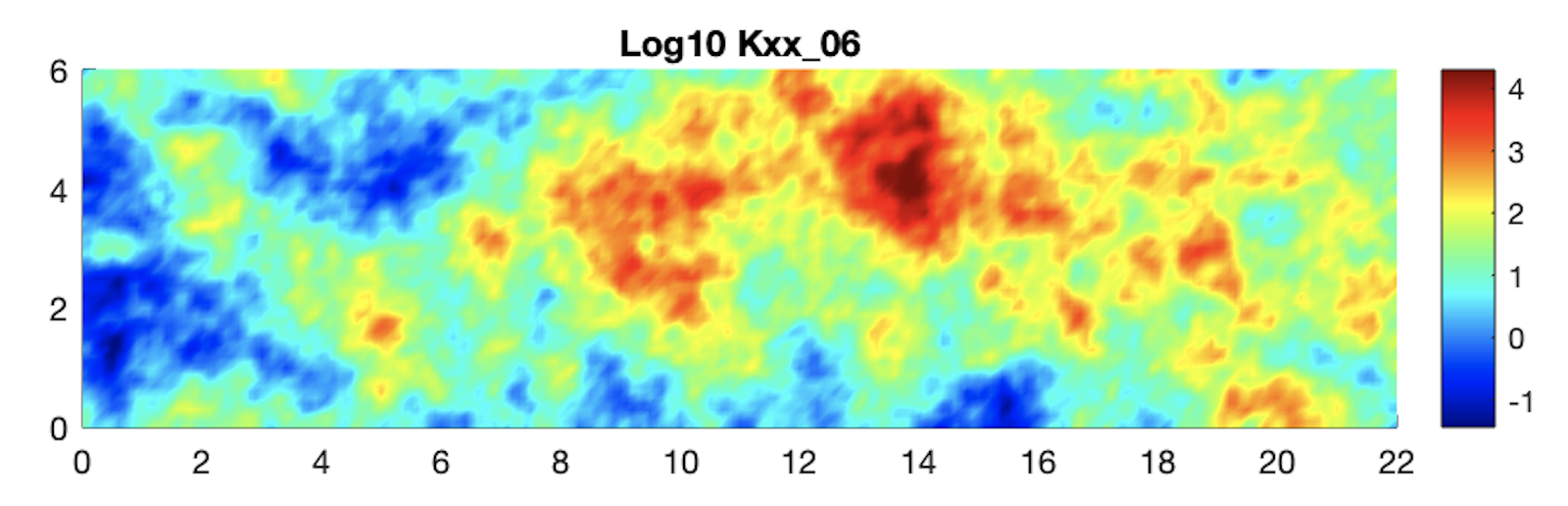}} \\
\subfloat[$Kxx\_85$]{\includegraphics[width=0.75\linewidth]{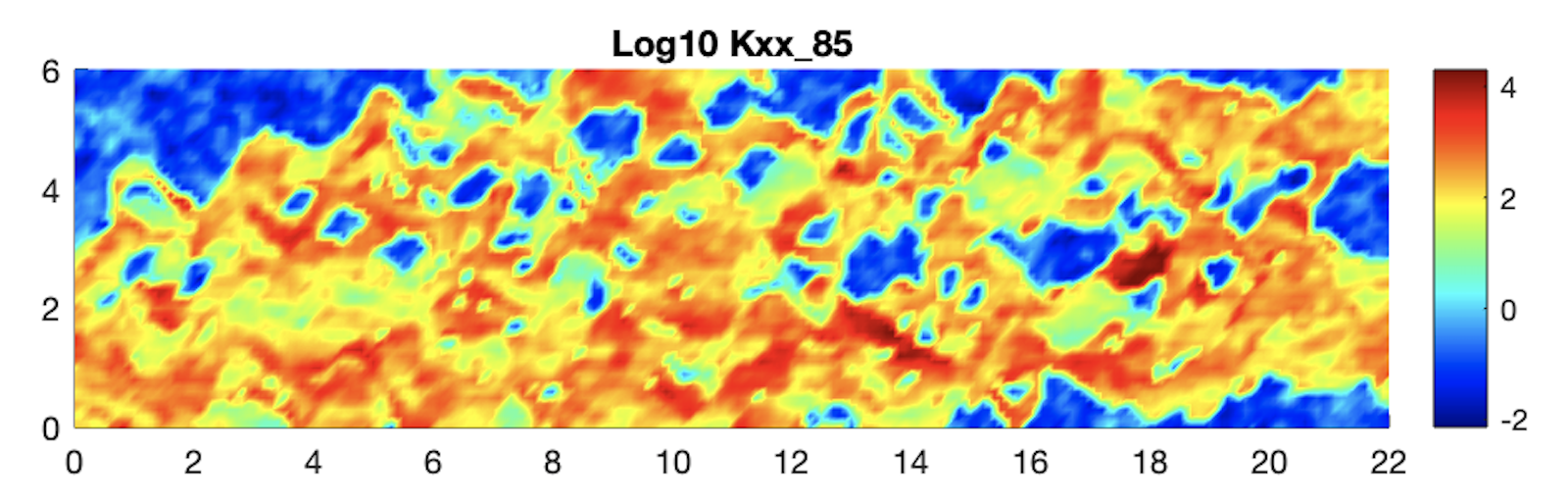}} 
\caption{The $\log_{10}$ values of coefficients for two types of SPE10 mesh.}
\label{NOSAS_figure8}
\end{figure}

\begin{table}[tbhp]
    \centering
   
\begin{tabular}{ |c|c|c|c| }
 \hline
$Kxx\_06$ & iterations & condition number & number of small eigenvalues \\
   \hline
  $\eta=0.25\frac{h}{H}$
  &94 & 123.18  & 26\\
   \hline
  $\eta=0.64\frac{h}{H}$ & 75 & 80.10 & 35\\
   \hline
  $\eta=1.60\frac{h}{H}$ & 53 & 33.72 & 66\\
\hline
\end{tabular}
\vspace*{3pt}\\
\begin{tabular}{ |c|c|c|c| }
 \hline
$Kxx\_85$ & iterations & condition number & number of small eigenvalues \\
   \hline
  $\eta=0.25\frac{h}{H}$
  &115 &  187.32  & 40\\
   \hline
  $\eta=0.64\frac{h}{H}$ & 77 & 84.88 & 64\\
   \hline
  $\eta=1.60\frac{h}{H}$ & 51 & 37.31 & 110\\
\hline
\end{tabular}
 \caption{Choosing different threshold $\eta$ for NOSAS with inexact diagonal solver in the mesh of SPE10 with 33 subdomains. }
 \label{NOSAS_table7}
\end{table}

\bibliographystyle{siamplain}
\bibliography{references}

\end{document}